\documentclass[twoside]{article}
\usepackage{amsthm,amsmath,amssymb}

\usepackage{geometry} 

\usepackage{setspace} 

\numberwithin{equation}{section} \numberwithin{figure}{section}
\theoremstyle{definition}
\newtheorem*{defn*}{Definition} \theoremstyle{plain}
\newtheorem{thm}{Theorem}
\newtheorem{lem}[thm]{Lemma}
\newtheorem{claim}[thm]{Claim}
\newtheorem*{lem*}{Lemma}
\newtheorem*{claim*}{Claim}
\newtheorem*{thm*}{Theorem} \theoremstyle{remark}
\newtheorem*{rem*}{Remark}
\numberwithin{thm}{section}

\begin{document}
\bibliographystyle{plain}

\title{Glauber Dynamics of colorings on trees}

\author{Allan Sly \thanks{University of California, Berkeley and Australian National University. Supported by an Alfred Sloan Fellowship and NSF grant DMS-1208338. Email:sly@stat.berkeley.edu} \and Yumeng Zhang \thanks{University of California, Berkeley. Email:zym3008@berkeley.edu} }

\date{}
\maketitle
\begin{abstract}
	The mixing time of the Glauber dynamics for spin systems on trees is closely related to reconstruction problem. Martinelli, Sinclair and Weitz established this correspondence for a class of spin systems with soft constraints bounding the log-Sobolev constant by a comparison with the block dynamics~\cite{martinelli2004glauber,martinelli2007fast}.
	However, when there are hard constraints, the block dynamics may be reducible.
	
	We introduce a variant of the block dynamics extending these results to a wide class of spin systems with hard constraints. This applies for essentially any spin system that has non-reconstruction provided that on average the root is not locally frozen in a large neighborhood. In particular we prove that the mixing time of the Glauber dynamics for colorings on the regular tree is $O(n\log n)$ in the entire known non-reconstruction regime.
\end{abstract}

\section{Introduction}

There has been substantial interest in understanding the rate of convergence of the Glauber dynamics for spin systems on trees and in particular how the mixing times scaling relates to the spatial mixing properties of the Gibbs measure. In the case of the coloring model, the natural conjecture is that there is rapid mixing (the mixing time is $O(n\log n)$) whenever the model is in the reconstruction regime. This was previously shown for the related block dynamics by Bhatnagar et al.\ \cite{bhatnagar2011reconstruction} but the presence of frozen local regions means their proof does not apply to the single site dynamics. We overcome this restriction establishing the following result.
\begin{thm}
	\label{thm:thm2}For fixed $\beta<1-\ln2$ and $k>k(\beta)$ large enough when $d\le k[\log k+\log\log k+\beta]$, the mixing time of Glauber dynamics of the $k$-coloring model on $n$-vertex $d$-ary tree is $O(n\log n)$.
\end{thm}

Our bound corresponds to the non-reconstruction region established in~\cite{sly2009reconstruction}. In the same paper it was shown that the $k$-coloring model is reconstructible for $\beta>1$. In a forthcoming work we will give an improved upper bound on the reconstruction threshold from which is will follow that $O(n\log n)$ mixing holds for the full non-reconstruction regime.

\subsection{Previous work}

Many studies have shown that the mixing time of Glauber dynamics, both for $k$-coloring model and general spin systems, are related to spatial properties of the Gibbs measure. In the case of $d$-ary trees, two properties of primary interest are uniqueness of infinite volume Gibbs measure and reconstruction, which corresponds to the extremality of infinite Gibbs measure induced by free boundary conditions. It has been shown in quite general settings that the Glauber dynamics exhibits rapid mixing, and in particular an $O(n\log n)$ mixing time, when the system admits a unique infinite volume Gibbs measure by Martinelli, Sinclair and Weitz~\cite{martinelli2004glauber,martinelli2007fast} and most generally in Weitz's thesis~\cite{weitz2004mixing}. Beyond the uniqueness threshold, the Glauber dynamics for $k$-coloring may not even be ergodic under some boundary conditions, but it is natural to consider the behavior under free boundary conditions.

There have been intensive studies on the mixing time of coloring model from both theoretical computer science and statistical physics. For $k$-coloring on general graphs with $n$ vertices and maximal degree $d$, the Glauber dynamics is irreducible when $k\ge d+2$. It is tempting to conjecture that the chain exhibits rapid mixing whenever $k\ge d+2$. So far the best result on general graphs is given by Vigoda in \cite{vigoda1999improved}, where he showed $O(n^{2}\log n)$ mixing time when $k\ge\frac{11}{6}d$. A series of improvements on the constant $\frac{11}{6}$ for rapid mixing have been made with extra conditions on the degree or  girth, (see the survey~\cite{frieze2007survey} for more results toward this direction).

For the $k$-coloring model on $d$-ary trees, Jonasson \cite{jonasson2002uniqueness} established there is a unique infinite Gibbs measure for $k\ge d+2$. Martinelli et al.\ \cite{martinelli2004glauber,martinelli2007fast, weitz2004mixing} studied the mixing rates of several spin systems on $d$-ary trees and in particular showed $O(n\log n)$ mixing time for coloring models under any boundary conditions for $k\ge d+3$. Their method made use of the block dynamics and employed the idea of the decay of correlation between the root and the leaves. As $k$ drops below the uniqueness threshold of $d+2$, we can no longer compare the block dynamics to the standard single site Glauber dynamics as the chain inside a block will not always be connected under worst-case boundary conditions. Notwithstanding this, the desired speed of decay of correlation still holds suggesting that rapid mixing under free boundary conditions might still be true beyond the non-uniqueness threshold.

This decay of correlation between root and leaves is closely related to another problem called reconstruction. Roughly speaking, a model on tree is reconstructible if, given the leaves of a randomly chosen configuration, one's best guess for the root is ``strictly better'' than the stationary distribution, as the number of level goes to infinity. In other words, reconstruction corresponds to the non-vanishing influence of average-case boundary conditions to the root. For coloring models, Mossel and Peres~\cite{MosPer:03} established reconstruction when $k\geq (1+o(1))d/\log d$ by considering the point at which the model freezes, that is the boundary condition exactly determines the root.

In the other direction, Bhatnagar et al.\ \cite{bhatnagar2011reconstruction} and Sly~\cite{sly2009reconstruction} independently proved that the model has non-reconstructible for $k\leq (1+o(1))d/\log d$. Using this result~\cite{bhatnagar2011reconstruction} showed that the block dynamics for $k$-coloring model mixes in $O(n\log n)$ time for $k\ge(1+\epsilon)d/\log d$ and large $k>k(\epsilon)$ using non-reconstruction and following the methods of \cite{martinelli2004glauber}.

For more results below the non-uniqueness threshold, Berger et al. \cite{berger2005glauber} showed polynomial mixing time for general models on trees whenever the dynamics is ergodic, which in the case of coloring corresponds to $k\ge 3$ and $d\ge 2$. For the coloring model Goldberg et al.\ \cite{goldberg2010mixing} proved an upper bound of $n^{O(d/\log d)}$ for the complete tree for with branching factor $d$ and Lucier et al.\ \cite{lucier2009glauber} showed the mixing time is $n^{O(1+d/k\log d)}$ for all $d$ and $k\ge 3$. Recently Tetali et al.~\cite{tetali2010phase} proved the mixing time undergoes a phase transition at the reconstruction threshold $k=(1+o(1))d/\log d$, where their upper bound for $k\ge(1+o(1))d/\log d$ is $O(n^{1+o_{k}(1)})$. They also showed that the mixing time is $\Omega(n^{d/k\log d-o_{k}(1)})$ for $k\le(1-o(1))d/\log d$, which implies rapid mixing does not hold in reconstruction region.

The main result of this paper is to reduce the mixing time in the non-reconstruction region from the polynomial time bound of $n^{1+o(1)}$ to a sharp bound of $O(n\log n)$. Our proof is a modification of the techniques used in \cite{martinelli2004glauber}. The main technical difficulty of directly applying their method is their (non-obvious) restriction of ``admissive'' and ``well-connected'', i.e. for any boundary conditions on arbitrary subset the set of proper configurations is non-empty and connected via single site update. These conditions may not be satisfied in models with hard constraints, in particular $k$-coloring model. Intuitively, below the uniqueness threshold, there will be vertices whose states are ``frozen'' by their neighbors. While block dynamics can update ``frozen'' vertices together with their neighbors in a single move, extra efforts are needed for single site dynamics to pass around the barrier and change them, leading to the failure of comparing two dynamics. We will look at a new variant of block dynamics that focuses on the connected component induced by the single site chain on the state space of usual block dynamics. By carefully examining the portion of ``frozen'' vertices and their influence on nearby sites, we will show rapid mixing of our new version of block dynamics which in turn implies the final result.

\subsection{General spin system}

Phase transitions and decay of correlation of spin systems play a key role in the mixing time of the Glauber dynamics. On trees the reconstruction threshold rather than the uniqueness threshold plays the key role. Berger et al.\ \cite{berger2005glauber} showed that for general spin systems, $O(n)$ relaxation time under free-boundary condition implies non-reconstruction. In the other direction, Weitz conjectured~\cite{weitz2004mixing} that for any $k$-state spin system or on $d$-ary trees, the system mixes in $O(n\log n)$ time whenever it admits a unique Gibbs measure and the Glauber dynamics is connected under given boundary condition. He proved for the case $k=2$ and for ferromagnetic Potts model and coloring as two special cases of $k>2$. His proof also gives spatial mixing conditions which apply for a wide range of models.

Our result for $k$-coloring can be extended to general $k$-state spin systems and gives a sufficient condition, for spin systems to exhibit rapid mixing in the non-reconstruction region. Throughout the paper, we will refer to a spin system by its probability kernel $M$, defined by $M(c,c')=\mu(\sigma_{y}=c'|\sigma_{x}=c),(x,y)\in E$. The sufficient condition, which we will call the {\bf connectivity condition} $\mathcal{C}$ is specified in Section~\ref{sub:cntcond} and deals with hard constraints. It roughly speaking asks that the probability that the root can ``change freely'' to all $k$ states given a random boundary condition tends to $1$. In particular, it is automatically satisfied by all models without hard constraints or with a permissive state -- one that can occur next to any other (e.g. the hardcore model).
\begin{thm}
	\label{thm:thm1} Let $M$ be a $k$-state system on the $n$-vertex $d$-ary tree $T$ with second eigenvalue $\lambda$. If $M$ satisfies the connectivity condition $\mathcal{C}$, is non-reconstructible on $T$, and $d\lambda^{2}<1$ then the mixing time of Glauber dynamics on $T$ under free boundary condition is $O(n\log n)$.
\end{thm}

The assumption that $d\lambda^{2}<1$ corresponds to the Kesten-Stigum bound in the reconstruction problem. When $d\lambda^{2}>1$, the system is always reconstructible by counting the number of leaves with each state (see, the survey~\cite{mossel2004survey}). Hence non-reconstruction implies $d\lambda^{2}\le 1$. Our assumption that $d\lambda^{2}$ is strictly less than 1 is essential to our proof and indeed must be so since Ding, Lubetzky and Peres~\cite{DLP:10} showed that the mixing time for the Ising model is at least of order $n\log^3 n$ when $d\lambda^2 = 1$.

\section{Preliminaries}

\subsection{Definition of model}

{\bf Spin systems: } We will use $[k]=\{1,\dots,k\}$ to denote the state space of a spin and $T=(V,E)$ to denote the $d$-ary tree (i.e. every vertex have $d$ offsprings) with root $\rho$ and $|T|=n$ vertices. Throughout the paper, we will denote the $l$-th level of the tree $T$ by $L_{l}$, with $L_{0}=\{\rho\}$. We will also use $T_{x}$ to represent the subtree rooted at $x\in T$ and let $B_{x,l}$, $L_{x,l}$ denote the first $l$ levels and the $l$-th level of $T_{x}$ respectively.

A configuration on $T$ is an assignment of spins to vertices $\sigma \in[k]^{V}$. For a $k$-state spin system (with potential $U$ and $W$), the probability of seeing $\sigma\in[k]^{V}$ is given by the (free-boundary) Gibbs measure
\[ \mu(\sigma)=\frac{1}{Z}\exp\bigg[-(\sum_{(x,y)\in E}U(\sigma_{x},\sigma_{y})+\sum_{x\in V}W(\sigma_{x}))\bigg], \]
where $U$ is a symmetric function from $[k]\times[k]\to\mathbb{R}\cup\{\infty\}$, $W$ is a function from $[k]\to\mathbb{R}\cup\{\infty\}$ and $Z$ is the partition function independent of $\sigma$ such that $\sum_{\sigma\in[k]^{V}}\mu(\sigma)=1$. If for some states $i,j\in[k]$, $U(i,j)=\infty$, we say $(i,j)$ is a hard constraint, otherwise we say $i$ and $j$ are \emph{compatible}. We will focus on the set of proper configurations on $T$, denoted by $\Omega_{T}=\{\sigma,\mu(\sigma)>0\}$. For example in the coloring model, $W(c)=0,U(c,c')=\infty\cdot1(c=c')$ and the Gibbs measure is the uniform distribution over all proper colorings. We will write $\sigma_{A}$ for the restriction of $\sigma$ to the subset $A$ and use superscript to denote conditioning on a boundary condition. $\Omega_{A}^{\eta}=\{\sigma,\sigma\in\Omega_{T},\sigma_{T\backslash A}=\eta_{T\backslash A}\}$ is the set of configurations compatible with boundary condition $\eta$ and the conditional law is $\mu_{A}^{\eta}(\sigma)=\mu(\sigma\mid\sigma\in\Omega_{A}^{\eta}).$

The principal example for this paper is the proper graph coloring. A proper $k$-coloring of graph $G=(V,E)$ is an assignment $\sigma:V\to[k]=\{1,2,\dots,k\}$ such that for all $(x,y)\in E$, $\sigma_{x}\neq\sigma_{y}$. In the statistical physics literature this corresponds to zero-temperature anti-ferromagnetic Potts model.

For the reconstruction problem, it is easy to work with the Markov chain construction of Gibbs measure on trees, which can be taken as a special case of the broadcast model on trees. We think of the process where information is sent on tree $T$ from the root $\rho$ downwards and each edge acts as a noisy channel. For each input $c_{1}\in[k]$, the output of $c_{2}$ is chosen from probability kernel $M(c_{1},c_{2})$. If the input at root $\rho$ follows the stationary distribution of $M$, denoted by $\pi$, the law of a random configuration on $T$ is given by
\[ \mu(\sigma)=\pi_{\sigma_{\rho}}\prod_{(x,y)\in E}M(\sigma_{x},\sigma_{y}). \]
It is easy to check the following one-to-one correspondence between potential $U,W$ and reversible probability kernel $M$,
\[ M(c_{1},c_{2})=\frac{\exp[-(U(c_{1},c_{2})+W(c_{2}))]}{\sum_{c'\in[k]}\exp[-(U(c_{1},c')+W(c'))]},U(c_{1},c_{2}) =\ln\left(\frac{M(c_{1},c_{2})}{\pi_{c_{2}}}\right),W(c)=-\ln\pi_{c}. \]
We will henceforth denote a spin system by its probability kernel $M$, we only deal with $M$ that are ergodic and reversible.

\subsubsection{Glauber dynamics and mixing time}

The Glauber dynamics for a $k$-state spin system $M$ is a Markov chain $X_{t}$ on state space $\Omega_{T}$. A step of the Markov chain from $X_{t}$ to $X_{t+1}$ is defined as follows:
\begin{enumerate}
	\item Pick a vertex $x$ uniformly randomly from $T$;
	\item Pick a state $c\in[k]$ according to the conditional distribution of the spin value of $x$ given the rest of configuration, i.e. state $c$ is picked with probability $\mu_{\{x\}}^{\sigma}(c)=\mu(\sigma_{x}'=c\mid\sigma'_{y}=\sigma_{y},\forall y\neq x)$;
	\item Set $X_{t+1}(x)=c$ and $X_{t+1}(y)=X_{t}(y),\forall y\neq x.$
\end{enumerate}
We denote the transition matrix by $P$. In the case of coloring, the second step corresponds to a picking a uniformly chosen color that does not appear in the neighbor of $x$.

As we will show in Lemma \ref{lem:freebc}, under the connectivity condition $\mathcal{C}$, the Glauber dynamics with free boundary conditions is ergodic, reversible and hence converges to the Gibbs measure $\mu$ (this is easy to check for models with no hard constraints). The {\bf mixing time} is defined as
\[ t_{\mathrm{mix}}=\max_{\sigma\in\Omega_{T}}\min\{t:d_{\mathrm{TV}}(P^{t}(\sigma,\cdot),\mu)\le\frac{1}{2e}\}, \]
where $P$ is the probability kernel of $X_{t}$ and $d_{\mathrm{TV}}(\pi,\mu)=\frac{1}{2}\sum_{\sigma}\left|\pi(\sigma)-\mu(\sigma)\right|$ is the total variance distance. To bound the mixing time we will make use of the {\bf log-Sobolev constant}. For non-negative function $f:\Omega_{T}\to\mathbb{R}$, let $\mu(f)=\sum_{\sigma}\mu(\sigma)f(\sigma)$ be the expectation of $f$ and the $\textrm{Ent}f=\mu(f\log f)-\mu(f)\log\mu(f)$ be the entropy. The Dirichlet form of $f$ is defined by
\[ \mathcal{D}(f)=\frac{1}{2}\sum_{\sigma,\sigma'}\mu(\sigma)P(\sigma,\sigma')(f(\sigma)-f(\sigma'))^{2}. \]
And the log-Sobolev constant is defined by $\gamma=\inf_{f\ge0}\frac{\mathcal{D}(\sqrt{f})}{\textrm{Ent}(f)}$. Applying results in functional analysis to Glauber dynamics yields the following bound on mixing time regarding log-Sobolev constant (see e.g. Theorem 2.2.5 of \cite{saloff1997lectures}):
\begin{thm*}
	For $k$-state system $M$ on $n$-vertex $d$-ary tree $T$, there exists a constant $C>0$ such that $t_{\mathrm{mix}}\le\frac{1}{\gamma}\cdot Cn\log n.$
\end{thm*}
Therefore to show rapid mixing is enough to show that $\gamma$ is uniformly bounded away from zero as $n$ tends to infinity.

\subsubsection{Uniqueness and reconstruction}

Two key notions of spatial decay of correlation for spin systems on trees are the uniqueness and reconstruction thresholds. Recalling that $L_{l}$ is the set of the vertices at level $l$ in $T$, we have the following definition
\begin{defn*}
	[Reconstruction] A $k$-state system $M$ is reconstructible on a tree $T$ if for some $c,c'\in[k]$
	\[ \limsup_{l\to\infty}d_{TV}(\mu(\sigma_{L_{l}}=\cdot\mid\sigma_{\rho}=c),\mu(\sigma_{L_{l}}=\cdot\mid\sigma_{\rho}=c'))>0. \]
	Otherwise we say the system has non-reconstruction on $T$.
\end{defn*}
Non-reconstruction is equivalent to the extermality of infinite volume Gibbs measure under free boundary conditions. More equivalent definition and an extensive literature are given in the survey~\cite{mossel2004survey}. A strictly stronger condition is the uniqueness property.
\begin{defn*}
	[Uniqueness] We say a $k$-state system $M$ has uniqueness on a tree $T$ if
	\[ \limsup_{l\to\infty}\sup_{\eta,\eta'\in\Omega_{L_{l}}}d_{TV}(\mu(\sigma_{\rho}=\cdot\mid\sigma_{L_{l}}=\eta),\mu(\sigma_{\rho}=\cdot\mid\sigma_{L_{l}}=\eta'))>0, \]
	where $\Omega_{L_{l}}$ is the set of configurations on level $l$.
\end{defn*}

\subsection{Connectivity condition\label{sub:cntcond}}

In this section we specify the connectivity condition $\mathcal{C}$. First we will define the notion for a vertex to be free. Let $T$ be a tree of $l$ levels. For configuration $\sigma\in\Omega_{T}$ with $\sigma_{\rho}=c$, $\sigma_{L_{l}}=\eta$, we say the root can change (from $c$) to state $c'$ {\bf in one step} if and only if there exist a path $\sigma=\sigma^{0},\sigma^{1},\dots,\sigma^{n}\in\Omega_{T}$ such that for each $i$, $\sigma^i,\sigma^{i+1}$ differ at exactly one vertex, $\sigma_{L_{l}}^i\equiv\eta$ for $\sigma_{\rho}^i=c,i=0,1,\dots.n-1,\sigma_{\rho}^{n}=c'$. Put another way, the path is a valid trajectory of the Glauber dynamics with fixed leaves which changes the state of $\rho$ only once in the final step. For $x\in T$, we say $x$ is {\bf free} (in $\sigma$) if, considered as root of $T_{x}$, $x$ can changed to all the other $(k-1)$-states in one step. Denote the probability that the root of an $l$-level tree is free as $p_{l}^{\mathrm{free}}=\mu(\sigma:\rho\text{ is free in }\sigma)$.
\begin{defn*}
	We say that the $k$-state system $M$ on the $d$-ary tree satisfies the connectivity condition $\mathcal{C}$ if $M$ is ergodic, reversible and
	\begin{enumerate}
		\item For all $ c_{1},c_{2},c_{3}\in[k]$, there exists $c\in[k]$ such that $c$ is compatible with $c_{1},c_{2},c_{3}$.
		\item The probability of being free tends to 1 as $l$ tends to infinity, i.e.~$\lim_{l\to\infty}p_{l}^{\mathrm{free}}=1.$
	\end{enumerate}
\end{defn*}
The first condition is used in the proof of Claim \ref{lem:freetoconnect}, see the discussion afterward for the necessity of this condition.

We first show that under connectivity condition $\mathcal{C},$ the Glauber dynamic is irreducible under free boundary conditions and ergodicity follows from that. For the sake of recursive analysis on subtrees later, we want to consider the case where the parent of the root is fixed to be some state. For state $c\in[k]$ , let $\Omega_{T}^{c}$ denote the set of configurations with the parent of root $\rho$ being state $c$ and let $\mu_{T}^{c}$ be the corresponding conditional Gibbs measure.
\begin{lem}
	\label{lem:freebc}For general $k$-state system $M$ on $d$-ary tree $T$. If $M$ is ergodic and reversible, then $\Omega_{T}^{c}$ is irreducible under Glauber dynamics.
\end{lem}
\begin{proof}
	Reversibility follows from the detailed balance equations. 
	We prove irreducibility by induction on the number of levels $l$ in $T$. For $l=0$, it is trivially true. We assume that the Glauber dynamics is connected for $(l-1)$-level tree and consider $l$-level tree $T$ and configurations $\sigma,\sigma'\in\Omega_{T}^{c}$. To establish a path of valid moves of the dynamics from $\sigma$ to $\sigma'$, one can first change every vertex $x\in L_{1}$ to state $c$ by a sequence of moves in the tree $T_x$ using our inductive assumption. We may then change the spin of the root from $\sigma_\rho$ to $\sigma_\rho'$. Finally we may change the configuration of every subtree $T_{x}$ to $\sigma'_{T_{x}}$ using the inductive assumption bringing us to the configuration $\sigma'$.
\end{proof}

\subsection{Component Dynamics }

In this section, we define a new variant of block dynamics on $T$, namely ``the component dynamics'', which updates a block of vertices each step, but only chooses configurations within the connected component of the Glauber dynamic. In this way we can utilize the techniques in \cite{martinelli2004glauber} while bypassing the problem that steps of block dynamic may not be connected in Glauber dynamics when $k\le d+1$. To give a formal definition, for $A\subset T$, we say $\sigma'\sim_{A}\sigma$ if $\sigma'{}_{T\backslash A}=\sigma_{T\backslash A}$ and $\sigma'_{A},\sigma_{A}$ are connected in Glauber dynamics on $A$ with fixed boundary condition $\sigma_{T\backslash A}$. We will omit the $A$ in $\sigma\sim_{A}\sigma'$ when it is clear from the context. Let $\Omega_{A}^{*,\sigma}=\{\sigma'\in\Omega_{A}^{\sigma},\sigma'\sim_{A}\sigma\}$ denote the connected component of $\sigma$ in $\Omega_{A}^{\sigma},$ and let $\mu_{A}^{*,\sigma}(\sigma')=\mu(\sigma'|\Omega_{A}^{*,\sigma})$ be the Gibbs distribution conditioned on both configuration outside $A$ and the connected component within $A$.

For $l\ge 1$, recall $B_{x,l}$ is the block of $l$ levels rooted at $x$ and $L_{x,l}$ be the $l$-th level of $B_{x,l}$. If $x$ is within distance $l$ of the leaves, let $B_{x,l}=T_{x}$. We define a step of the {\bf component dynamics} by the update rule:
\begin{enumerate}
	\item Pick a vertex $x$ uniformly randomly from $T$,
	\item Replace $\sigma$ by $\sigma'$ drawn from conditional distribution $\mu_{B_{x,l}}^{*,\sigma}$.
\end{enumerate}
The dynamics is reversible with respect to the Gibbs distribution. For test function $f:\Omega_{T}\to\mathbb{R}$, let $\mu_{A}^{*,\sigma}(f)=\sum_{\sigma'\in\Omega_{A}^{*,\sigma}}f(\sigma')\mu_{A}^{*,\sigma}(\sigma')$ be the conditional expectation of $f$ on $\Omega_{A}^{*,\sigma}$ and for $f\ge 0$, let
\[ \mathrm{Ent}_{A}^{*,\sigma}(f)=\mathrm{Ent}(f\mid\Omega_{A}^{*,\sigma})=\mu_{A}^{*,\sigma}(f\log f)-\mu_{A}^{*,\sigma}(f)\log\mu_{A}^{*,\sigma}(f) \]
be the conditional entropy of $f$. We write the sum of local entropies of block size $l$ as $\mathcal{E}_{l}^{*}=\sum_{x\in T}\mu_{T}(\mathrm{Ent}_{B_{x,l}}^{*,\sigma}(f))$. With minor modification, the comparison result of block dynamics also works for component dynamics: (see e.g. Prop 3.4 of \cite{martinelli1999lectures}, in the proof substitute $\mathcal{E}_{D}(f,f)$ by $\mathcal{E}_{l}^{*}$ and note $\sum_{\sigma'}\mu_{T}^{\tau}(\sigma')\mu_{B_{x,l}}^{*,\sigma'}(\sigma)=\mu_{T}^{\tau}(\sigma)$.)
\[ \gamma\ge\frac{1}{l}\cdot\inf_{f\ge0}\frac{\mathcal{E}_{l}^{*}}{\mathrm{Ent}(f)}\cdot\min_{\sigma,x}\gamma_{Bx,l}^{*,\sigma} \]
where $\gamma_{Bx,l}^{*,\sigma}$ is the log-Soblev constant of Glauber dynamics on $\Omega_{B_{x,l}}^{*,\sigma}$ with boundary condition on $
\partial B_{x,l}$ given by $\sigma$. From our definition of $\Omega_{B_{x,l}}^{*,\sigma}$, it is easy to see that $\min_{\sigma,x}\gamma_{Bx,l}^{*,\sigma}$ is a constant only depending on the branching number $d$, block size $l$ and $M$ itself and is strictly greater than 0 independent of $T$. Thus to show $O(n\log n)$ mixing time for Glauber dynamics, it is enough to show $\mathcal{E}_{l}^{*}\ge\mbox{const}\times\mathrm{Ent}(f)$ for all $f\ge0$ and some choice of block size $l$ independent of tree size $|T|=n$.

\subsection{Outline of Proof}

A key ingredient in \cite{martinelli2004glauber} is that a certain strong concentration property implies ``entropy mixing'' in space which in turn implies the fast mixing of block dynamics. The following Theorem \ref{thm:Dror5.3} can be seen as the combination of Theorems 3.4 and~5.3 of \cite{martinelli2004glauber} adapted to component dynamics (the notation here is closer to Theorem 5.1 of \cite{bhatnagar2011reconstruction}). For completeness, we include an outline of the proof in Section \ref{sec:ProofofDror} pointing out the differences from the original argument.
\begin{thm}
	\label{thm:Dror5.3}There exist some constant $\alpha>0$ such that for every $\delta>0$ and $l\geq 1$, if for all $x\in T$ that is at least $l$ levels from the leaves and any compatible pair of states $c,c'\in[k]$, $\mu^{c}=\mu_{T_{x}}^{c}$ satisfies
	\begin{equation}
		\Pr\nolimits_{\tau\sim\mu^{c}}\left(\left|\frac{\mu^{c}(\sigma_{x}=c'\mid\sigma\sim_{B_{x,l}}\tau)}{\mu^{c}(\sigma_{x}=c')}-1\right|\ge\frac{(1-\delta)^{2}}{\alpha(l+1-\delta)^{2}}\right)\le e^{-2\alpha(l+1-\delta)^{2}/(1-\delta)^{2}}\label{eq:newcct}
	\end{equation}
	then for every function $f\ge0$, $\mathrm{Ent}(f)\le\frac{2}{\delta}\mathcal{E}_{l}^{*}.$
\end{thm}

To prove Theorem \ref{thm:thm1}, it suffices to verify (\ref{eq:newcct}) for some choice of $l$ and $\delta$. Note the original inequality in Theorem~5.3 of \cite{martinelli2004glauber} or Theorem~5.1 of~\cite{bhatnagar2011reconstruction} is
\begin{equation}
	\Pr\nolimits_{\tau\sim\mu^{c}}\left(\left|\frac{\mu^{c}(\sigma_{x}=c'\mid\sigma_{L_{x,l}}=\tau_{L_{x,l}})}{\mu^{c}(\sigma_{x}=c')}-1\right|\ge\frac{(1-\delta)^{2}}{\alpha(l+1-\delta)^{2}}\right)\le e^{-2\alpha(l+1-\delta)^{2}/(1-\delta)^{2}}\label{eq:originalcct}
\end{equation}

The only difference between (\ref{eq:newcct}) and (\ref{eq:originalcct}) is that in equation (\ref{eq:newcct}), the inner measure $\mu^{c}$ conditions not only on boundary condition $\sigma_{L_{x,l}}=\tau_{L_{x,l}}$, but also the connected component of $\tau$. We will first prove a stronger version the original inequality given the non-reconstruction of system $M$ and $d\lambda^{2}<1$.
\begin{thm}
	\label{lem:originalcct}For a $k$-state system $M$, if $M$ is non-reconstructible and $d\lambda^{2}<1$, then there exist some constant $\xi>0$, $l_{0}\ge1$ such that for all $l\ge l_{0}$, every $x\in T$ that is at least $l$ levels from the leaves, and any compatible pair of states $c,c'\in[k]$, $\mu^{c}=\mu_{T_{x}}^{c}$ satisfies
	\begin{equation}
		\Pr\nolimits_{\tau\sim\mu^{c}}\left(\left|\frac{\mu^{c}(\sigma_{x}=c'\mid\sigma_{L_{x,l}}=\tau_{L_{x,l}})}{\mu^{c}(\sigma_{x}=c')}-1\right|\ge e^{-\xi l}\right)\le\exp(-e^{\xi l}).\label{eq:dbexpdecay-1}
	\end{equation}
\end{thm}
From there, we will show under our connectivity condition, the difference between $\sigma\sim_{B_{x,l}}\tau$ and $\sigma_{L_{x,l}}=\tau_{L_{x,l}}$ is negligible in the upper half of the block when $l$ is large and hence a similar tail distribution holds for the root.
\begin{lem}
	\label{lem:newcct}Let $M$ be a $k$-state system satisfying $\mathcal{C}$ such that (\ref{eq:originalcct}) holds for $l\ge l_{0}$ and $\delta=\delta_0$. Then there exist constants $l_1\ge 2l_{0}$ and $\delta_1\ge\delta_0$ such that for all $l\ge l_1$, equation (\ref{eq:newcct}) holds with $\delta=\delta_1$.
\end{lem}

Theorems~\ref{thm:Dror5.3} and~\ref{lem:originalcct} and Lemma~\ref{lem:newcct} together imply Theorem~\ref{thm:thm1}. The rest of the paper is structured as follows, we will prove Theorem \ref{lem:originalcct} in Section \ref{sec:doubleexpdecay} and Lemma \ref{lem:newcct} in Section \ref{sec:twocct}, and we will include a sketch of Theorem \ref{thm:Dror5.3} in Section \ref{sec:ProofofDror}. After that we will apply the result to the $k$-coloring model and prove Theorem~\ref{thm:thm2} in Section \ref{sec:coloring}.

\section{From non-reconstruction to the strong concentration property\label{sec:doubleexpdecay}}

In this section we prove Theorem \ref{lem:originalcct}. The result for $k$-coloring model was proved in \cite{bhatnagar2011reconstruction}, which used the specific structure of coloring model. Here we will give a different proof for general systems $M$ using only non-reconstruction and that $M$ is strictly below Kesten-Stigum bound $d\lambda^{2}<1$. We first introduce some notations. Recall that the stationary distribution of $M$ is $\pi$. For $x\in T$, let
\[ \tilde{R}_{x,l}(\tau)(c)=\frac{1}{\pi_{c}}\mu_{T_{x}}(\sigma_{x}=c\mid\sigma_{L_{x,l}}=\tau_{L_{x,l}}) \]
denote the ratio of conditional and unconditional distribution at $x$ and write $R_{x,l}(\tau)=\|\tilde{R}_{x,l}(\tau)-1\|_\infty=\max_{c\in[k]}|\tilde{R}_{x,l}(\tau)(c)-1|$. We will omit $\tau$ when it is clear from context.
In the proof we will work with the unconditional Gibbs measure $\mu=\mu_{T_{x}}$ and $\pi$ instead of $\mu_{T_{x}}^{c}$ and $\mu_{T_{x}}^{c}(\sigma_{x}=c')$ and show the following inequality
\begin{equation}
	\Pr\nolimits_{\tau\sim\mu}\left(R_{x,l}(\tau)\ge e^{-\xi l}\right)\le\exp(-e^{\xi l}).\label{eq:dbexpdecay}
\end{equation}
To see (\ref{eq:dbexpdecay}) implies (\ref{eq:dbexpdecay-1}), consider the Markov chain construction of $\sigma$. Let $E$ be the edge set of $T_{x}$, we have
\[ \mu(\sigma)=\pi_{\sigma_{x}}\prod_{(y,z)\in E}M(\sigma_{y},\sigma_{z}),\ \mu^{c}(\sigma)=M(c,\sigma_{x})\prod_{(y,z)\in E}M(\sigma_{y},\sigma_{z}). \]
Hence the Radon-Nikodym derivative $\frac{d\mu^{c}}{d\mu}$ satisfies $\frac{d\mu^{c}}{d\mu}(\sigma)=\frac{M(c,\sigma_{x})}{\pi_{\sigma_{x}}}\le \pi^{-1}_{\min}$, where $\pi_{\min}=\min_{c\in[k]}\pi_{c}$. Note
\[ \left|\frac{\mu^{c}(\sigma_{x}=c'\mid\sigma_{L_{x,l}}=\tau_{L_{x,l}})}{\mu^{c}(\sigma_{x}=c')}-1\right| = \left|\frac{1}{\pi_{c'}}\mu(\sigma_{x}=c'\mid\sigma_{L_{x,l}}=\tau_{L_{x,l}})-1\right| \le R_{x,l}(\tau), \]
we have that the LHS of equation (\ref{eq:dbexpdecay-1}) can be upper bounded by the LHS of (\ref{eq:dbexpdecay}) up to a factor of $\pi_{\min}^{-1}$.
The following lemma gives the recursive relation of $\tilde{R}_{x,l}(c)$.
\begin{lem}
	Let $x_{1},\dots,x_d$ denote the $d$ children of $x\in T$, $\tilde{R}_{x,l}$ can be written as a function of $\tilde{R}_{x_i,l-1}$:
	\begin{equation}
		\tilde{R}_{x,l}(c)=\frac{\prod_{i=1}^dM\tilde{R}_{x_i,l-1}}{\pi\prod_{i=1}^dM\tilde{R}_{x_i,l-1}}(c) =\frac{\prod_{i=1}^d\sum_{c_i\in[k]}M(c,c_i)\tilde{R}_{x_i,l-1}(c_i)} {\sum_{c'\in[k]}\pi_{c'}\prod_{i=1}^d\sum_{c_i\in[k]}M(c',c_i)\tilde{R}_{x_i,l-1}(c_i)}.\label{eq:Bayesrecur}
	\end{equation}
\end{lem}
\begin{proof}
	Let $E$ and $E_i$ denote the edge set of $T_{x}$ and $T_{x_i}$, they satisfy $E=\cup_i\left(E_i\cup\{(x,x_i)\}\right)$. Let $\Omega(c)=\{\sigma:\sigma_{x}=c,\sigma_{L_{x,l}}=\tau_{L_{x,l}}\}$ and $\Omega_i(c)=\{\sigma:\sigma_{x_i}=c,\sigma_{L_{x_i,l-1}}=\tau_{L_{x_i,l-1}}\}$ be the set of configurations on $T_{x}$ and $T_{x_i}$ with boundary condition $\tau$. By the Markov chain construction, we have
	\begin{align*}
		\mu(\Omega(c)) & =\mu(\sigma_{x}=c,\sigma_{L_{x,l}}=\tau_{L_{x,l}})\\
		& =\sum_{\sigma\in\Omega(c)}\pi_{c}\prod_{(y,z)\in E}M(\sigma_{y},\sigma_{z})=\sum_{c_{1},\cdots,c_d\in[k]}\pi_{c}\prod_{i=1}^dM(c,c_i)\sum_{\sigma^i\in\Omega_i(c_i)}\prod_{(y,z)\in E_i}M(\sigma_{y}^i,\sigma_{z}^i)\\
		& =\sum_{c_{1},\cdots,c_d\in[k]}\pi_{c}\prod_{i=1}^d\frac{M(c,c_i)}{\pi_{c_i}}\mu(\Omega_i(c_i))=\pi_{c}\prod_{i=1}^d\sum_{c_i\in[k]}\frac{M(c,c_i)}{\pi_{c_i}}\mu(\Omega_i(c_i)).
	\end{align*}
	Therefore by Bayes formula,
	\begin{align*}
		\tilde{R}_{x,l}(c) & =\frac{1}{\pi_{c}}\mu(\sigma_{x}=c\mid\sigma_{L_{x,l}}=\tau_{L_{x,l}})=\frac{1}{\pi_{c}}\frac{\mu(\Omega(c))}{\sum_{c'\in[k]}\mu(\Omega(c'))}\\
		& =\frac{\prod_{i=1}^d\sum_{c_i\in[k]}\frac{M(c,c_i)}{\pi_{c_i}}\mu(\Omega_i(c_i))}{\sum_{c'\in[k]}\pi_{c'}\prod_{i=1}^d\sum_{c_i\in[k]}\frac{M(c',c_i)}{\pi_{c_i}}\mu(\Omega_i(c_i))}=\frac{\prod_{i=1}^d\sum_{c_i\in[k]}M(c,c_i)\tilde{R}_{x_i,l-1}(c_i)}{\sum_{c'\in[k]}\pi_{c'}\prod_{i=1}^d\sum_{c_i\in[k]}M(c',c_i)\tilde{R}_{x_i,l-1}(c_i)}.
	\end{align*}
	where the last step followed by dividing both the nominator and denominator by $\prod_{i=1}^d\sum_{c'_i\in[k]}\mu(\Omega_i(c'_i))$.
\end{proof}
Observe that in the recursive relationship of (\ref{eq:Bayesrecur}), $\tilde{R}_{x,l}(c)$ is a rational function of $(\tilde{R}_{x_i,l-1}(c_i))_{i=1,\dots,d,c_i\in[k]}$. If for all $i=1,\dots,d$ we have $\tilde{R}_{x_i,l-1}=1$, then $\tilde{R}_{x,l}=1$. By the continuity of (\ref{eq:Bayesrecur}) and ergodicity of $M$, we establish the following contraction property of $R_{x,l}$.
\begin{lem}
	\label{lem:Rcontraction}There exist an integer $m\ge 1$ and constant $\epsilon>0$ such that for all $d^{m}$ vertices $y_{1},\dots,y_{d^{m}}\in L_{x,m}$, if at most one $y_i$ has $R_{y_i,l-m}>\epsilon$ then
	\begin{equation}
		R_{x,l}\leq\frac{1}{2}\sum_{i=1}^{d^{m}}R_{y_i,l-m}.\label{eq:Rcontraction}
	\end{equation}
\end{lem}
\begin{proof}
	Denote the range of $\tilde{R}_{x,l}$ by the $k$ dimensional simplex $\Delta_{[k]}=\{R\in\mathbb{R}^{k},\pi R=1,R_i\ge0,i=1,\dots,k\}$. Let $f:\Delta_{[k]}^d\to\Delta_{[k]}$ be the function on the RHS of (\ref{eq:Bayesrecur}) such that $\tilde{R}_{x,l}=f(\tilde{R}_{x_{1},l-1},\dots,\tilde{R}_{x_d,l-1})$.
	When $\tilde{R}_{x_{2},l-1}=\cdots=\tilde{R}_{x_d,l-1}=1$, the function can be simplified as
	\[ \tilde{R}_{x,l}=f(\tilde{R}_{x_{1},l-1},1,\dots,1)=\frac{M\tilde{R}_{x_{1},l-1}}{\pi M\tilde{R}_{x_{1},l-1}}=M\tilde{R}_{x_{1},l-1}. \]
	Iterating the function $m$ times, we can write $\tilde{R}_{x,l}=f^{(m)}(\tilde{R}_{y_{1},l-m},\dots,\tilde{R}_{y_{d^{m}},l-m})$ for some function $f^{(m)}:\Delta_{[k]}^{d^{m}}\to\Delta_{[k]}$ 	
	. A similar calculation shows when $\tilde{R}_{y_{2},l-m}=\cdots=\tilde{R}_{y_{d^{m}},l-m}=1$,
	\[ \tilde{R}_{x,l}=f^{(m)}(\tilde{R}_{y_{1},l-m},1,\dots,1)=M^{m}\tilde{R}_{y_{1},l-m}. \]

	Since $f^{(m)}$ is smooth, for some $C_1=C_1(d,m,M)$ we have that
	\[ \| \tilde{R}_{x,l} - 1 - \sum_{i=1}^{d^{m}} (M^{m}\tilde{R}_{y_{i},l-m} - 1) \| \le C_1 \sum_{i=1}^{d^{m}} \|\tilde{R}_{y_{i},l-m}-1\|^2 \le C_1 k\sum_{i=1}^{d^{m}} \|\tilde{R}_{y_{i},l-m}-1\|_\infty^2. \]
	By the ergodicity of $M$, for sufficiently large $m$ and all $\tilde{R}\in\Delta_{[k]}$ we have $\|M^{m}\tilde{R} - 1\|_\infty \leq \frac{1}{4} \|\tilde{R} - 1\|_\infty$. Therefore for some $\epsilon_1=\epsilon_1(C_1,k)$ if $R_{y_i,l-m}\le\epsilon_{1}$ for all vertices $y_i\in L_{x,m}$ then
	\begin{equation}
		\label{eq:contractionA} \|\tilde{R}_{x,l}-1\|_\infty \le(\frac{1}{4}+C_1k\epsilon_1)\sum_{i=1}^{d^{m}}\|\tilde{R}_{y_{i},l-m}-1\|_\infty \le\frac{1}{2}\sum_{i=1}^{d^{m}}R_{y_i,l-m}.
	\end{equation}
	This suffices provided that there are no large $R_{y_i,l-m}$.	
	
	We now consider the case when there is one large $R_{y_i,l-m}$, which we can without loss of generality assume is $i=1$. 	
	Again since $f^{(m)}$ is smooth, there exists $C_2,\epsilon_{2}>0$ such that for all $\tilde{R}_{y_{1},l-m}>\epsilon_1$, if $\sup_{i\geq 2} R_{y_i,l-m}\le\epsilon_{2}$ then 	
	\[ \| \tilde{R}_{x,l} - M^{m}\tilde{R}_{y_{1},l-m} \| \leq C_2 \sum_{i=2}^{d^{m}} \|\tilde{R}_{y_{i},l-m}-1\|. \]
	Let $ \epsilon=\epsilon_2\wedge (4C_2 d^m k)^{-1}\epsilon_1 $, if we moreover have $\sup_{i\geq 2} R_{y_i,l-m}\le\epsilon$, then
	\begin{equation}
		\label{eq:contractionB} \|\tilde{R}_{x,l}-1\|_\infty \le\frac{1}{4}\|\tilde{R}_{y_1,l-m}-1\|_\infty + C_2 d^m k\epsilon \le \frac {1}{4} R_{y_1,l-m} + \frac{1}{4}\epsilon_1 \le \frac{1}{2} R_{y_1,l-m}.
	\end{equation}
	Combining equations~\eqref{eq:contractionA} and~\eqref{eq:contractionB} and noting $ \epsilon<\epsilon_1$ completes the proof.	

\end{proof}
So far we have not used the assumption of non-reconstruction and $d\lambda^{2}<1$. In~\cite{janson2004robust}, Janson and Mossel develop a notion called robust reconstruction and show that if $M$ satisfies $d\lambda^{2}<1$ then there exist constant $C_{1}>0$ depending on $d\lambda^{2}$ and $\delta>0$ depending on $d$ and $C_{1}$ such that if for some $l$, $d_{TV}(\mu_{L_{l}}^{c},\mu_{L_{l}})\le\delta$ for all $c\in[k]$, then $d_{TV}(\mu_{L_{l+1}}^{c},\mu_{L_{l+1}})\le e^{-C_{1}}d_{TV}(\mu_{L_{l}}^{c},\mu_{L_{l}})$. In our case, the existence of such $l$ is guaranteed by the definition of non-reconstruction, hence we have for some constant $C_{2}$ that
\[ d_{TV}(\mu_{L_{l}}^{c},\mu_{L_{l}})\le C_{2}e^{-C_{1}l}. \]
A duality argument shows
\begin{align*}
	\mathbb{E}_{\tau\sim\mu}|\tilde{R}_{x,l}(c)-1| & =\mathbb{E}_{\tau\sim\mu}\left|\frac{1}{\pi_{c}}\mu(\sigma_{x}=c\mid\sigma_{L_{x,l}}=\tau_{L_{x,l}})-1\right |=\mathbb{E}_{\tau\sim\mu}\left|\frac{\mu^{c}(\sigma_{L_{x,l}}=\tau_{L_{x,l}})}{\mu(\sigma_{L_{x,l}}=\tau_{L_{x,l}})}-1\right|\\
	& =\sum_{\tau}\left|\mu^{c}(\sigma_{L_{x,l}}=\tau_{L_{x,l}})-\mu(\sigma_{L_{x,l}}=\tau_{L_{x,l}})\right|=2d_{TV}(\mu_{L_{l}}^{c},\mu_{L_{l}})\le 2C_{2}e^{-C_{1}l}.
\end{align*}
Maximizing over $c$ we get $\mathbb{E}_{\tau\sim\mu}R_{x,l}\le C_{2}e^{-C_{1}l}$ for some (different) constant $C_{1},C_{2}>0$ and by Markov's inequality for all $l\ge 1,z>0$,
\begin{equation}
	\Pr\nolimits_{\tau\sim\mu}(R_{x,l}>z)\le\frac{C_{2}}{z}e^{-C_{1}l}.\label{eq:Svantebound}
\end{equation}
\begin{proof}
	[Proof of Theorem \ref{lem:originalcct}] By Lemma \ref{lem:Rcontraction}, the event $R_{x,l}>z$ implies that either there exist two $i$ such that $R_{y_i,l-m}>\epsilon$ or $\sum_{i=1}^{d^{m}}R_{y_i,l-m}>2z$. In the second case if the event $\sum_{i=1}^{d^{m}}R_{y_i,l-m}>2z$ holds and for every $y_i$, $R_{y_i,l-m}\le\frac{3}{2}z$, then there must exist at least two $i$ such that $R_{y_i,l-m}>\frac{1}{2d^{m}}z$, otherwise $\sum_{i=1}^{d^{m}}R_{y_i,l-m}\le\frac{3}{2}z+\frac{d^{m}-1}{2d^{m}}z<2z$. Therefore we can write
	\begin{align*}
		\Pr\nolimits_{\tau\sim\mu}(R_{x,l}>z) & \le\Pr\nolimits_{\tau\sim\mu}(\exists\mbox{ two }y_i\in L_{x,m},R_{y_i,l-m}>\epsilon)+\Pr\nolimits_{\tau\sim\mu}(\exists y_i\in L_{x,m},R_{y_i,l-m}>\frac{3}{2}z)\\
		& \hphantom{\le}+\Pr\nolimits_{\tau\sim\mu}(\exists\mbox{ two }y_i\in L_{x,m},R_{y_i,l-m}>\frac{1}{2d^{m}}z).
	\end{align*}
	Let $g(z,l)=\Pr\nolimits_{\tau\sim\mu}(R_{x,l}>z)$ and $C=\max\{2d^{m},\frac{1}{\epsilon\pi_{\min}}\}$, note $g(z,l)$ is a decreasing function in $z$, the equation above become
	\begin{align*}
		g(z,l) & \le d^{2m}g^{2}(\epsilon,l-m)+d^{m}g(\frac{3}{2}z,l-m)+d^{2m}g^{2}(\frac{1}{2d^{m}}z,l-m)\\
		& \le d^{m}g(\frac{3}{2}z,l-m)+2d^{2m}g^{2}(\frac{1}{C}z,l-m).
	\end{align*}
	Iterating this estimation $h$ times, we have
	\begin{equation}
		g(z,l)\le\sum_{i=0}^{h}(2d^{2m})^{2^{h-i}(1+i)}g^{2^{h-i}}\left((\frac{3}{2})^i(\frac{1}{C})^{h-i}z,l-hm\right).\label{eq:Grecursion}
	\end{equation}
	where the coefficient can be shown by induction on $h$ using inequality $(a+b)^{2}\le2(a^{2}+b^{2})$.
	
	Since for all $z>\pi_{\min}^{-1}$ we have $g(z,l)=0$, the summand on the RHS of (\ref{eq:Grecursion}) is zero for large $i.$ 	
	Fix $\kappa=\log\frac{4}{3}C/\log\frac{3}{2}C<1$, for $h\ge\log(\frac{1}{z\pi_{\min}})/\log(\frac{4}{3})$ and $i>\kappa h$, we have $(\frac{3}{2})^i(\frac{1}{C})^{h-i}z>\pi_{\min}^{-1}.$ Therefore
	\begin{align*}
		g(z,l) & \le\sum_{i=0}^{\kappa h}(2d^{2m})^{2^{h-i}(1+i)}g^{2^{h-i}}\left((\frac{3}{2})^i(\frac{1}{C})^{h-i}z,l-hm\right)\le\kappa h\left[(2d^{2m})^{h}g\left(C^{-h}z,l-hm\right)\right]^{2^{(1-\kappa)h}}
	\end{align*}
	Now applying the bound in (\ref{eq:Svantebound}) and let $h=rl/m$ for small $r>0$ such that $(1-r)C_{1}-r\cdot\frac{1}{m}\log(2Cd^{2m})>\frac{1}{2}C_{1}>0$, for large enough $l$ such that $ \log l \le 2^{\frac{(1-\kappa)r}{m}l} $ we have
	\begin{align*}
		g(z,l) & \le\kappa h\left((2d^{2m})^{h}\frac{C_{2}C^{h}}{z}e^{-C_{1}(l-hm)}\right)^{2^{(1-\kappa)h}}\\
		& \le\frac{\kappa r}{m}l\left(\frac{C_{2}}{z}(2Cd^{2m})^{\frac{r}{m}l}e^{-C_{1}(1-r)l}\right)^{2^{\frac{(1-\kappa)r}{m}l}} \le\frac{\kappa r}{m}\left(\frac{2C_{2}}{z}e^{-\frac{1}{2}C_{1}l}\right)^{2^{\frac{(1-\kappa)r}{m}l}}.
	\end{align*}
	Let $C_{3}=\frac{2}{C_{1}},C_{4}=\frac{\kappa r}{m}$, $C_{5}=\frac{(1-\kappa)r}{m}\log2$, for $l>C_{3}(1+\log 2C_{2}-\log z)$, we have $g(z,l)\le C_{4}\exp\{-\exp(C_{5}l)\}$.
	
	Finally define $\xi=\frac{1}{2}\min\{C_{3}^{-1},C_{5}\}$, plug in $z_{l}=\exp(-\xi l)$. When $l$ is large enough, we have $C_{3}(1+\log 2C_{2}-\log z)\le C_{3}(1+\log 2C_{2})+\frac{1}{2}l<l$ and $\exp(\exp(\frac{1}{2}C_{5}l))>C_{4}$, therefore
	\[ \Pr\nolimits_{\tau\sim\mu}\left(R_{x,l}(\tau)\ge e^{-\xi l}\right)=g(z_{l},l)\le C_{4}\exp(-e^{C_{5}l})\le\exp(-e^{\xi l}), \]
	completing the proof.
\end{proof}

\section{Proof of Lemma \ref{lem:newcct} \label{sec:twocct}}

The proof of Lemma \ref{lem:newcct} uses a two step analysis. First for block $B_{x,l}$ with sufficiently large $l$, we consider the measure $\mu_{B_{x,l}}^{*,\tau}$ induced on the upper half of block $B_{x,l/2}$. Denote the following subset of $\Omega_{B_{x,l}}^{\tau}$,
\[ A_{\tau}=\{\sigma\in\Omega_{B_{x,l}}^{\tau},\forall x\in L_{x,l/2+2},x\mbox{ is free w.r.t. }\sigma\}. \]
$A_{\tau}$ can be considered as the set of ``good'' configurations with boundary condition $\tau$. As we will show later, under connectivity condition $\mathcal{C}$, $\mu_{B_{x,l}}^{*,\tau}(A_{\tau})$ is close to $1$ with high probability. And as the following lemma claims, conditioning on $A_{\tau}$ and the configuration on boundary $L_{x,l/2}$, the measure induced by $\mu_{B_{x,l}}^{*,\tau}$ actually equals to $\mu^{c}$. Therefore, as a second step we can apply the result of Lemma \ref{lem:originalcct} to $B_{x,l/2}$. Let $\Omega_{L_{x,l/2}}$ be the set of possible configuration on $L_{x,l/2}$.
\begin{lem}
	\label{lem:condindep}Suppose $M$ satisfies the connectivity condition $\mathcal{C}$. For arbitrary $\tau\in\Omega_{T_{x}}^{c}$, $\eta\in\Omega_{L_{x,l/2}}$ and state $c'\in[k]$ that is compatible with $c$,
	\begin{equation}
		\mu_{B_{x,l}}^{*,\tau}(\sigma_{x}=c'\mid\sigma_{L_{x,l/2}}=\eta,\sigma\in A_{\tau})=\mu^{c}(\sigma_{x}=c'\mid\sigma_{L_{x,l/2}}=\eta).\label{eq:condindep}
	\end{equation}
\end{lem}
\begin{proof}
	For convenience of notation, abbreviate $\sigma_{(1)}=\sigma_{B_{x,l/2-1}}$, $\sigma_{(2)}=\sigma_{B_{x,l}\backslash B_{x,l/2}}$, so every configuration $\sigma\in\Omega_{B_{x,l}}$ can be written as a three tuple $(\sigma_{(1)},\eta,\sigma_{(2)})$. We of course have that $\sigma_{(1)},\sigma_{(2)}$ are conditionally independent given $\sigma_{L_{x,l/2}}=\eta$. By the definition of $A_{\tau}$, $\{\sigma\in A_{\tau}\}$ only depends on $\sigma_{(2)}$. Therefore to show (\ref{eq:condindep}), it is enough to show that conditioned on $\sigma_{L_{x,l/2}}$ and $\sigma\in A_{\tau}$, $\sigma\sim\tau$ is independent of $\sigma_{(1)}.$ From there we have
	\begin{align*}
		\mu_{B_{x,l}}^{*,\tau}(\sigma_{x}=c'\mid\sigma_{L_{x,l/2}}=\eta,\sigma\in A_{\tau}) & =\mu^{c}(\sigma_{x}=c'\mid\sigma_{L_{x,l/2}}=\eta,\sigma\sim\tau,\sigma\in A_{\tau})\\
		& =\mu^{c}(\sigma_{x}=c'\mid\sigma_{L_{x,l/2}}=\eta,\sigma\in A_{\tau})=\mu^{c}(\sigma_{x}=c'\mid\sigma_{L_{x,l/2}}=\eta).
	\end{align*}
	Since ``$\sim$'' is a transitive relation, the conditional independence of $\sigma\sim\tau$ and $\sigma_{(1)}$ follows from the following claim.
\end{proof}
\begin{claim}
	\label{lem:freetoconnect} Suppose $M$ satisfies the connectivity condition $\mathcal{C}$. For each $\tau\in\Omega_{T_{x}}^{c}$ and $\eta\in\Omega_{L_{x,l/2}}$, and for all $\sigma=(\sigma_{(1)},\eta,\sigma_{(2)})$, $\sigma'=(\sigma'_{(1)},\eta,\sigma_{(2)})\in\Omega_{B_{x,l}}^{\tau}$ if $\sigma,\sigma'\in A_{\tau}$, then $\sigma\sim\sigma'$.
\end{claim}
\begin{proof}
	By Lemma \ref{lem:freebc}, there exist a path $\Gamma$ connecting $\sigma_{(1)}$ to $\sigma'_{(1)}$ in $\Omega_{B_{x,l/2}}^{c}$, i.e. the Glauber dynamics on $B_{x,l/2}$ with free boundary condition on leaves and the state of $p(x)$ fixed to be $c$. We will construct a path $\Gamma'$ in $\Omega_{B_{x,l}}^{\tau}$ connecting $\sigma$ to $\sigma'$ by adding steps between steps of $\Gamma$ which only changes the configuration on $B_{x,l}\backslash B_{x,l/2}$, such that vertices in $L_{x,l/2+1}$ won't block the move in $\Gamma$ and after finishing $\Gamma$, we can change the configuration on $B_{x,l}\backslash B_{x,l/2}$ back to the original $\sigma_{(2)}$. The construction of $\Gamma'$ is specified below:
	
	(1) {\bf Before starting $\Gamma$.} For each $y\in T$, let $p(y)$ denote the parent of $y$. For each $y\in L_{x,l/2+2}$, $\sigma\in A_{\tau}$ implies that there exist a path $\Gamma_{y}$ in $T_{y}$ changing $y$ from $\sigma_{y}$ to $\sigma_{p(p(y))}=\eta_{p(p(y))}$ in one step. To see $\Gamma_{y}$ is also a connected path in $B_{x,l}$, we have to show that the parent of $y$ won't block $\Gamma_{y}$. The only neighbor of $p(y)$ in $T_{y}$ is $y$ and the only move involving $y$ in $\Gamma_{y}$ is the last step changing $y$ from $\sigma_{y}$ to $\sigma_{p(p(y))}$. The value of $p(y)$ won't block this last step because $\sigma_{p(y)}$ is compatible with both $\sigma_{y}$ and $\sigma_{p(p(y))}$ (they are states of neighbouring vertices in $\sigma$). Now we will follow each $\Gamma_{y}$ for all $y\in L_{x+l/2+2}$ and change $\sigma_{y}$ to $\sigma_{p(p(y))}.$ After that, for each $w\in L_{x,l/2}$, all vertices in $L_{w,2}$ are in state $\sigma_{w}=\eta_{w}$. The configuration on and below $L_{x,l/2+2}$ will henceforth remain fixed until we finish $\Gamma$.
	
	(2) {\bf Performing $\Gamma$}. For each step in $\Gamma$, the existence of $B_{x,l-1}\backslash B_{x,l/2}$ might block this move only if it changes the state of some vertex $w\in L_{x,l/2}$. Suppose it is changes $w$ from $c_{1}$ to $c_{2}$, remember in the construction above, all vertices in $L_{w,2}$ have states $\eta_{w}$. By part 1 of $\mathcal{C}$, we can find $c_{3}\in[k]$ which is compatible with $c_{1},c_{2}$ and $\eta_{w}$. Now in order to change $w$ from $c_{1}$ to $c_{2}$, it suffices to first change the state of every vertex $z\in L_{w,1}$ to $c_{3}$, and then change $w$ from $c_{1}$ to $c_{2}$. This construction keeps the configuration on and below $L_{x,l/2+2}$ unchanged.
	
	(3) {\bf After $\Gamma$.} The configuration in $B_{x,l/2}$ is $(\sigma'_{(1)},\eta)$. We can change every vertex $z\in L_{x,l/2+1}$ back to $\sigma'_{z}=\sigma_{z}$ because at this moment its parent $p(z)\in L_{x,l/2}$ and all children of $z$ in $L_{z,1}$ have state $\eta_{p(z)}=\sigma_{p(z)}$, which is compatible with $\sigma_{z}$. From there, we can reverse the path $\Gamma_{y}$ for each $y\in L_{x,l/2+2}$ and change the configuration on and below $L_{x,l/2+2}$ back to the original configuration $\sigma_{(2)}$. This completes the construction achieving $\sigma'_{(2)}=\sigma_{(2)}$.
\end{proof}

Note Claim \ref{lem:freetoconnect} combined with Lemma \ref{lem:cnctcond} below implies that, with high probability (i.e.~on $A_{\tau}$), the fixed boundary Glauber dynamics on $B_{x,l/2}$ is actually connected as a subgraph of the Glauber dynamic on larger block $B_{x,l}$. This is one part of the proof where connectivity condition is used. We may replace the present connectivity condition by a more general assumption that the probability that the fixed boundary Glauber dynamics on $B_{x,l}$ is not connected in a larger block $B_{x,l'}$ decays double exponentially fast in block size $l$.

\begin{lem}
	\label{lem:cnctcond} The connectivity condition $\mathcal{C}$ implies that there exist constants $C_{1}>1$, $C_{2}>0$ such that for all $l\ge 1$,
	\begin{equation}
		1-p_{l}^\mathrm{free}\le C_{2}\exp(-C_{1}^{l}).\label{eq:cnctcond}
	\end{equation}
\end{lem}
\begin{proof}
	Fix $x\in T$ and $\sigma\in\Omega_{T_{x}}$. First if for all $1\leq i \leq d$, $z_i\in L_{x,1}$ is free, then $x$ is also free. To see that, for any $ c\in[k]$, by connectivity condition there exists $c'\in[k]$ such that $c'$ is compatible with both $c$ and $\sigma_{x}$, we can first change all $z_i$ to $c'$ in one step and then change $x$ from $\sigma_{x}$ to $c$ as the final step.
	
	Now consider $y_{ij}\in L_{z_i,1}\subset L_{x,2}$ for $1\leq i,j \leq d$. If at most one of the $y_{ij}$'s is not free, say $y_{11}\in L_{z_{1},1}$, then for $i\neq1$, $z_i$ is free and $z_{1}$ can change in one step to all states compatible with $\sigma_{y_{11}}$. Again by $\mathcal{C}$, for all $c\in[k]$ there exists $c'\in[k]$ such that $c'$ is compatible with $c$, $\sigma_{x}$ and $\sigma_{y_{11}}$. By the construction above, we can change $x$ from $\sigma_{x}$ to $c$ in one step, hence $x$ is also free.
	
	This implies if $x$ is not free, then there exist at least two $y_{ij}\in L_{x,2}$ that are not free. By part 2 of $\mathcal{C}$, there exists $l_{0}>0$, such that for all $l>l_{0}$ we have $1-p_{l}^\mathrm{free}<\frac{1}{d^{8}}$ and hence
	\[ 1-p_{l}^\mathrm{free}\le\binom{d^{2}}{2}(1-p_{l-2}^\mathrm{free})^{2}\le d^{4}(1-p_{l-2}^\mathrm{free})^{2}\le(1-p_{l-2}^\mathrm{free})^{1.5}. \]
	By induction, $1-p_{l}^\mathrm{free}\le(1-p_{l_{0}}^\mathrm{free})^{(1.5)^{(l-l_{0})/2}}$ which completes the proof.
\end{proof}

Now we can finish the proof of Lemma \ref{lem:newcct} from which Theorem \ref{thm:thm1} follows immediately.
\begin{proof}
	[Proof of Lemma \ref{lem:newcct}] Let $C=\alpha(l/2+1-\delta)^{2}/[(1-\delta)^{2}\mu^{c}(\sigma_{x}=c')]$ be the number on the left hand side of (\ref{eq:originalcct}). It is enough to show for some constant $l_1\ge2l_{0}$, and some coefficient $K\ge1$, for all $l\ge l_1$
	\[ \Pr\nolimits_{\tau\sim\mu^{c}}\left(\left|\mu^{c}(\sigma_{x}=c'\mid\sigma\sim\tau)-\mu^{c}(\sigma_{x}=c')\right|\ge\frac{K}{C}\right)\le e^{-2C/K}. \]
	To see the sufficiency, note this is just equation (\ref{eq:newcct}) with $\delta_1$ satisfying $1-\delta_1=\frac{1}{4K}(1-\delta)$.
	
	Recall $A_{\tau}=\{\sigma\in\Omega_{B_{x,l}}^{\tau} : \forall x\in L_{x,l/2+2},x\mbox{ is free in }\sigma\}$. Lemma \ref{lem:cnctcond} implies that for some constant $C_{1}>1$, $C_{2}>0$, and $l\ge1$
	\begin{align*}
		\mathbb{E}_{\tau\sim\mu^{c}}(\mu_{B_{x,l}}^{*,\tau}(A_{\tau}^{c})) & =\mathbb{E}_{\tau\sim\mu^{c}}(\mu^{c}(\sigma\notin A_{\tau}\mid\sigma\sim\tau))\\
		& =\Pr\nolimits_{\sigma\sim\mu^{c}}(\exists y\in L_{x,l/2+2},y\mbox{ is not free })\\
		& \le d^{l/2+2}(1-p_{l/2-2}^\mathrm{free})\le C_{2}d^{l/2+2}\exp(-C_{1}^{l/2-2}).
	\end{align*}
	By Markov inequality,
	\begin{equation}
		\Pr\nolimits_{\tau\sim\mu^{c}}\left(\mu_{B_{x,l}}^{*,\tau}(A_{\tau}^{c})>\frac{1}{2C}\right)\le2C\mathbb{E}_{\tau\sim\mu^{c}}(\mu_{B_{x,l}}^{*,\tau}(A_{\tau}^{c}))\le Cd^{l/2+2}C_{2}\exp(-C_{1}^{l/2-2})\to0,\label{eq:hppart1}
	\end{equation}
	as $l\to \infty$. On the event $\{\tau : \mu_{B_{x,l}}^{*,\tau}(A_{\tau}^{c})\le\frac{1}{2C}\}$,
	\begin{align}
		\mu_{B_{x,l}}^{*,\tau}(\sigma_{x}=c'\mid\sigma\in A_{\tau}) & =\frac{\mu_{B_{x,l}}^{*,\tau}(\sigma_{x}=c',\sigma\in A_{\tau})}{\mu_{B_{x,l}}^{*,\tau}(\sigma\in A_{\tau})}\le\frac{1}{1-\frac{1}{2C}}\mu_{B_{x,l}}^{*,\tau}(\sigma_{x}=c',\sigma\in A_{\tau})\nonumber \\
		& \le(1+\frac{1}{C})\left(\mu_{B_{x,l}}^{*,\tau}(\sigma_{x}=c')+\frac{1}{C}\right)\le\mu_{B_{x,l}}^{*,\tau}(\sigma_{x}=c')+\frac{3}{C}\nonumber \\
		\mu_{B_{x,l}}^{*,\tau}(\sigma_{x}=c'\mid\sigma\in A_{\tau}) & \ge\mu_{B_{x,l}}^{*,\tau}(\sigma_{x}=c',\sigma\in A_{\tau})\ge\mu_{B_{x,l}}^{*,\tau}(\sigma_{x}=c')-\frac{1}{C}.\label{eq:condonAtau}
	\end{align}
	Combining the two results together we have
	\begin{equation}
		\left|\mu_{B_{x,l}}^{*,\tau}(\sigma_{x}=c')-\mu_{B_{x,l}}^{*,\tau}(\sigma_{x}=c'\mid\sigma\in A_{\tau})\right|\le\frac{3}{C}.\label{eq:unbiapart2}
	\end{equation}
	
	Now splitting $\mu_{B_{x,l}}^{*,\tau}(\sigma_{x}=c'\mid\sigma\in A_{\tau})$ according to $\sigma_{L_{x,l/2}}$ and applying Lemma \ref{lem:condindep}, we have
	\begin{align}
		\mu_{B_{x,l}}^{*,\tau}(\sigma_{x}=c'\mid\sigma\in A_{\tau}) & =\sum_{\eta}\mu_{B_{x,l}}^{*,\tau}(\sigma_{x}=c'\mid\sigma\in A_{\tau},\sigma_{L_{x,l/2}}=\eta)\mu_{B_{x,l}}^{*,\tau}(\sigma_{L_{x,l/2}}=\eta\mid\sigma\in A_{\tau})\nonumber \\
		& =\sum_{\eta}\mu^{c}(\sigma_{x}=c'\mid\sigma_{L_{x,l/2}}=\eta)\mu_{B_{x,l}}^{*,\tau}(\sigma_{L_{x,l/2}}=\eta\mid\sigma\in A_{\tau}).\label{eq:condoneta}
	\end{align}
	We would like to estimate the set of $\eta$ such that $\mu^{c}(\sigma_{x}=c'\mid\sigma_{L_{x,l/2}}=\eta)$ has a large bias. Let $B=\{\eta : \left|\mu^{c}(\sigma_{x}=c'\mid\sigma_{L_{x,l/2}}=\eta)-\mu^{c}(\sigma_{x}=c')\right|\ge\frac{1}{C}\}$, Theorem \ref{lem:originalcct} implies that for $l/2\ge l_{0}$ and some $\delta>0$, we have $\Pr\nolimits_{\eta\sim\mu^{c}}\left(B\right)\le e^{-2C},$ where $\eta\sim\mu^{c}$ means the induced measure on $L_{x,l/2}$. Again by Markov's inequality,
	\begin{equation}
		\Pr\nolimits_{\tau\sim\mu^{c}}(\mu_{B_{x,l}}^{*,\tau}(\sigma_{L_{x,l/2}}\in B)>\frac{1}{C})\le C\mathbb{E}_{\tau\sim\mu^{c}}\mu_{B_{x,l}}^{*,\tau}(\sigma_{L_{x,l/2}}\in B)=C\mu^{c}(B)\le Ce^{-2C}.\label{eq:hppart2}
	\end{equation}
	On the event $\{\tau:\mu_{B_{x,l}}^{*,\tau}(\sigma_{L_{l/2}}\in B)\le\frac{1}{C}\}\cap\{\tau:\mu_{B_{x,l}}^{*,\tau}(A_{\tau}^{c})\le\frac{1}{2C}\}$, from (\ref{eq:condoneta}) we have
	\begin{align}
		\left|\mu_{B_{x,l}}^{*,\tau}(\sigma_{x}=c'\mid\sigma\in A_{\tau})-\mu^{c}(\sigma_{x}=c')\right| & \le\sum_{\eta}\left|\mu^{c}(\sigma_{x}=c'\mid\sigma_{L_{l/2}}=\eta)-\mu^{c}(\sigma_{x}=c')\right|\mu_{B_{x,l}}^{*,\tau}(\sigma_{L_{l/2}}=\eta\mid\sigma\in A_{\tau})\nonumber \\
		& \le\sum_{\eta\in B^{c}}\frac{1}{C}\mu_{B_{x,l}}^{*,\tau}(\sigma_{L_{l/2}}=\eta\mid\sigma\in A_{\tau})+\mu_{B_{x,l}}^{*,\tau}(\sigma_{L_{l/2}}\in B\mid\sigma\in A_{\tau})\nonumber \\
		& \le\frac{1}{C}\cdot 1+\frac{1}{C}+\frac{3}{C}=\frac{5}{C}\label{eq:unbiapart3}
	\end{align}
	where the last inequality follows from similar argument to (\ref{eq:condonAtau}).
	
	Combining the result of equations (\ref{eq:unbiapart2}) and (\ref{eq:unbiapart3}), on the event $\{\tau:\mu_{B_{x,l}}^{*,\tau}(\sigma_{L_{l/2}}\in B)\le\frac{1}{C}\}\cap\{\tau:\mu_{B_{x,l}}^{*,\tau}(A_{\tau}^{c})\le\frac{1}{2C}\}$, we have
	\[ \left|\mu_{B_{x,l}}^{*,\tau}(\sigma_{x}=c')-\mu^{c}(\sigma_{x}=c')\right|\le\frac{3}{C}+\frac{5}{C}=\frac{8}{C}. \]
	Therefore using the bounds from (\ref{eq:hppart1}) and (\ref{eq:hppart2}), for all $l\ge 2l_{0}$,
	\begin{align*}
		\Pr\nolimits_{\tau\sim\mu^{c}}\left(\left|\mu^{c}(\sigma_{x}=c'\mid\sigma\sim\tau)-\mu^{c}(\sigma_{x}=c')\right|>\frac{8}{C}\right) & \le\Pr\nolimits(\mu_{B_{x,l}}^{*,\tau}(\sigma_{L_{l/2}}\in B)\le\frac{1}{C})+\Pr\nolimits(\mu_{B_{x,l}}^{*,\tau}(A_{\tau}^{c})\le\frac{1}{2C})\\
		& \le Cd^{l/2+2}C_{2}\exp(-C_{1}^{l/2-2})+Ce^{-2C}\le e^{-16C}.
	\end{align*}
	where recall that $C\approx C'l^{2}$ for some constant $C'$ depending on $\delta$ and $\alpha$, the last step is true for some large enough constant $\tilde{l}$ depending on $d$, $C_{1}$, $C_{2}$ and $C'$. This means the strong concentration inequality (\ref{eq:newcct}) holds, for $K=8$, $\delta_1=1-\frac{1}{4K}(1-\delta)$ and $l_1=\max\{2l_{0},\tilde{l}\}$. Moreover, by taking $l$ large enough and changing the constant $C$ to $8C$ in (\ref{eq:unbiapart2}) and (\ref{eq:unbiapart3}), we can make $K$ arbitrarily close to $1$.
\end{proof}

\section{Component dynamics version of fast mixing results \label{sec:ProofofDror}}

In this section we prove Theorem \ref{thm:Dror5.3}. The theorem was originally proved for block dynamics in \cite{martinelli2004glauber}. Here we give a modification of their theorem adapted to the component dynamics by roughly ``adding stars'' at all occurrence of $B_{x,l}$. We will only state the key steps and refer the details to \cite{martinelli2004glauber}. For the remainder of this section, we let $\mu=\mu_{T}^{c},\Omega=\Omega_{T}^{c}$. Recall that $\tilde{T}_{x}=T_{x}\backslash\{x\}$. First we define the entropy mixing condition for Gibbs measure to be the following:
\begin{defn*}
	[Entropy Mixing] We say that $\mu$ satisfies $\mathrm{EM}^{*}(l,\epsilon)$ if for every $x\in T$, $\eta\in\Omega$ and any $f\ge 0$ that does not depend on the connected component of $B_{x,l}$, i.e. $f(\sigma)=\mu_{B_{x,l}}^{*,\sigma}(f),\forall\sigma\in\Omega$, we have $\mathrm{Ent}_{T_{x}}^{\eta}[\mu_{\tilde{T}_{x}}(f)]\le\epsilon\cdot\mathrm{Ent}_{T_{x}}^{\eta}(f)$ where $\mathrm{Ent}_{T_{x}}^{\eta}$ means the entropy w.r.t $\mu_{T_{x}}^{\eta}$.
\end{defn*}
Denote $p_{\min}=\min_{c,c'\in[k]}\{M(c,c'):M(c,c')>0\}$. By the Markov chain construction of configurations, we have $p_{\min}=\min_{x,c,c'}\{\mu_{T_{x}}^{c}(\sigma_{x}=c'),c,c'\mbox{ are compatible}\}$. The following theorem relates entropy mixing condition to the log-Soblev constant.

\begin{thm}\label{thm:Dror3.4}
	For any $l$ and $\delta>0$, if $\mu$ satisfies $\mathrm{EM}^{*}(l,[(1-\delta)p_{\min}/(l+1-\delta)]^{2})$ then $\mathrm{Ent}(f)\le\frac{2}{\delta}\cdot\mathcal{E}_{l}^{*}(f).$
\end{thm}
To show Theorem \ref{thm:Dror3.4}, we need the following modification of Lemma 3.5 (ii) of \cite{martinelli2004glauber}. The proof follows from its analog in \cite{martinelli2004glauber} immediately once we replace $\nu_{A},\mathrm{Ent}_{A},\nu_{B},\mathrm{Ent}_{B}$ there with $\nu_{\tilde{T}_{x}},\mathrm{Ent}_{\tilde{T}_{x}},\nu_{B_{x,l}}^{*},\mathrm{Ent}_{B_{x,l}}^{*}$ respectively.
\begin{lem}\label{lem:Dror3.5}
	For any $\epsilon<p_{\min}^{2},$ if $\mu$ satisfies $\mathrm{EM}^{*}(l,\epsilon)$ then for every $x\in T$, any $\eta\in\Omega$ and any $f\ge 0$ we have $\mathrm{Ent}_{T_{x}}^{\eta}[\mu_{\tilde{T}_{x}}(f)]\le\frac{1}{1-\epsilon'}\cdot\mu_{T_{x}}^{\eta}[\mathrm{Ent}_{B_{x,l}}^{*}(f)]+\frac{\epsilon'}{1-\epsilon'}\cdot\mu_{T_{x}}^{\eta}[\mathrm{Ent}_{\tilde{T_{x}}}(f)]$ with $\epsilon'=\sqrt{\epsilon}/p_{\min}.$
\end{lem}
Now plugging $\epsilon=[(1-\delta)p_{\min}/(l+1-\delta)]^{2}$ into Lemma \ref{lem:Dror3.5} verifies the hypothesis of the following claim, which then implies Theorem \ref{thm:Dror3.4}:
\begin{claim}
	If for every $x\in T$, $\eta\in\Omega$ and any $f\ge 0$,
	\begin{equation}
		\mathrm{Ent}_{T_{x}}^{\eta}[\mu_{\tilde{T}_{x}}(f)]\le c\cdot\mu_{T_{x}}^{\eta}[\mathrm{Ent}_{B_{x,l}}^{*}(f)]+\frac{1-\delta}{l}\cdot\mu_{T_{x}}^{\eta}[\mathrm{Ent}_{\tilde{T_{x}}}(f)],\label{eq:entmixinghyp}
	\end{equation}
	then $\mathrm{Ent}(f)\le\frac{c}{\delta}\cdot\mathcal{E}_{l}^{*}(f)$ for all $f\ge 0$.
\end{claim}
\begin{proof}
	First we decompose $\mathrm{Ent}(f)$ as a sum of $\mathrm{Ent}_{T_{x}}^{\eta}[\mu_{\tilde{T}_{x}}(f)]$. Suppose $T$ have $m$ levels, consider $\varnothing=F_{0}\subset F_{1}\subset\cdots\subset F_{m+1}=T$, where $F_i$ is the lowest $i$ levels of $T$. By basic properties of conditional entropy (equation (3), (4), (5) of \cite{martinelli2004glauber}) and Markov property of Gibbs measure, we have
	\begin{align}
		\mathrm{Ent}(f) =\dots=\sum_{i=1}^{m+1}\mu[\mathrm{Ent}_{F_i}(\mu_{F_{i-1}}(f))] \le\sum_{i=1}^{m+1}\sum_{x\in F_i\backslash F_{i-1}}\mu[\mathrm{Ent}_{T_{x}}(\mu_{F_{i-1}}(f))]\le\sum_{x\in T}\mu[\mathrm{Ent}_{T_{x}}(\mu_{\tilde{T}_{x}}(f))].\label{eq:entdecomp}
	\end{align}
	Denote the final sum by $\mathrm{PEnt}(f)$. For each term in the sum of $\mathrm{PEnt}(f)$, apply (\ref{eq:entmixinghyp}) to $g=\mu_{T_{x}\backslash B_{x,l}\cup
	\partial B_{x,l}}(f)$ and perform the decomposition trick of (\ref{eq:entdecomp}) again, we have for every $x\in T$ and $\eta\in\Omega$.
	\begin{align*}
		\mathrm{Ent}_{T_{x}}^{\eta}[\mu_{\tilde{T}_{x}}(f)] & =\mathrm{Ent}_{T_{x}}^{\eta}[\mu_{\tilde{T}_{x}}(g)]\le c\cdot\mu_{T_{x}}^{\eta}[\mathrm{Ent}_{B_{x,l}}^{*}(g)]+\frac{1-\delta}{l}\cdot\mu_{T_{x}}^{\eta}[\mathrm{Ent}_{\tilde{T}_{x}}(g)]\\
		& \le c\cdot\mu_{T_{x}}^{\eta}[\mathrm{Ent}_{B_{x,l}}^{*}(f)]+\frac{1-\delta}{l}\cdot\sum_{y\in B_{x,l}\cup
		\partial B_{x,l},y\neq x}\mu_{T_{x}}^{\eta}[\mathrm{Ent}_{T_{y}}(\mu_{\tilde{T}_{y}}(f))].
	\end{align*}
	Now sum up for all $x\in T$ and take expectation w.r.t $\mu$ for $\eta\in\Omega$, noting that the first term of the last line sums up to $\mathcal{E}_{l}^{*}=\sum_{x\in T}\mu(\mathrm{Ent}_{B_{x,l}}^{*}(f))$ and each $y$ in second term appears in at most $l$ blocks so we get
	\begin{align*}
		\mathrm{PEnt}(f) & \le c\cdot\mathcal{E}_{l}^{*}(f)+\frac{1-\delta}{l}\cdot\sum_{x\in T}\sum_{y\in B_{x,l}\cup
		\partial B_{x,l},y\neq x}\mu[\mathrm{Ent}_{T_{y}}(\mu_{\tilde{T}_{y}}(f))]\\
		& \le c\cdot\mathcal{E}_{l}^{*}(f)+\frac{1-\delta}{l}\cdot l\cdot\sum_{y\in T}\mu[\mathrm{Ent}_{T_{y}}(\mu_{\tilde{T}_{y}}(f))]=c\cdot\mathcal{E}_{l}^{*}(f)+(1-\delta)\cdot\mathrm{PEnt}(f),
	\end{align*}
	and hence $\mathrm{Ent}(f)\le\mathrm{PEnt}(f)\le\frac{c}{\delta}\cdot\mathcal{E}_{l}^{*}.$
\end{proof}
Given the result of Theorem \ref{thm:Dror3.4}, it is enough to show that for some constant $\alpha$, the super concentration inequality of (\ref{eq:newcct}) implies $\mathrm{EM}^{*}(l,[(1-\delta)p_{\min}/(l+1-\delta)]^{2})$. For convenience of notation, define following two functions for each $c'\in[k]$:
\[ g_{c'}(\sigma)=\frac{\mu(\sigma|\sigma_{\rho}=c')}{\mu(\sigma)}=\frac{1}{\mu(\sigma_{\rho}=c')}\cdot 1_{\left\{ \sigma_{\rho}=c'\right\} },g_{c'}^{*(l)}=\mu_{B_{\rho,l}}^{*}(g_{c'}). \]
Letting $\delta'=(1-\delta)^{2}/\alpha(l+1-\delta)^{2}$, we can rewrite (\ref{eq:newcct}) as
\begin{equation}
	\mu\left(\left|g_{c'}^{*(l)}-1\right|>\delta'\right)\le e^{-2/\delta'}.\label{eq:newcctshort}
\end{equation}
\begin{thm}\label{thm:Dror5.3true}
	There exists a constant $C$ such that if (\ref{eq:newcctshort}) holds for some $\delta'\ge0$ and all pairs of states $c,c'\in[k]$, we have $\mathrm{Ent}[\mu_{\tilde{T}}(f)]\le C\delta'\mathrm{Ent}(f)$ for any $f\ge0$ satisfying $f(\sigma)=\mu_{B_{\rho,l}}^{*,\sigma}(f),\forall\sigma\in\Omega^{c}$, i.e. $\mathrm{EM}^{*}(l,C\delta')$ holds.
\end{thm}
\begin{proof}
	Since for any $f'\ge0$, $\mathrm{Ent}(f')\le\mathrm{Var}(f')/\mu(f)$, we can write
	\begin{align}
		\mathrm{Ent}[\mu_{\tilde{T}}(f)] & \le\frac{\mathrm{Var}[\mu_{\tilde{T}}(f)]}{\mu(\mu_{\tilde{T}}(f))}=\frac{1}{\mu(f)}\sum_{c'\in[k]}\mu(\sigma_{\rho}=c')\left(\mu(f|\sigma_{\rho}=c')-\mu(f)\right)^{2}\nonumber \\
		& =\frac{1}{\mu(f)}\sum_{c'\in[k]}\mu(\sigma_{\rho}=c')\mathrm{Cov}(g_{c'},f)^{2}\le\max_{c'\in[k]}\frac{\mathrm{Cov}(g_{c'},f)^{2}}{\mu(f)}=\max_{c'\in[k]}\frac{\mathrm{Cov}(g_{c'}^{*(l)},f)^{2}}{\mu(f)}.\label{eq:cov}
	\end{align}
	where covariance is taken w.r.t. $\mu$ and the last step is because $f(\sigma)=\mu_{B_{\rho,l}}^{*,\sigma}(f)$. Now using Lemma 5.4 of \cite{martinelli2004glauber} (cited below) with
	\[ f_{1}=\frac{g_{c'}^{*(l)}-1}{\left\Vert g_{c'}^{*(l)}\right\Vert _{\infty}},f_{2}=\frac{f}{\mu(f)} \]
	and noting that $\left\Vert g_{c'}^{*(l)}\right\Vert _{\infty}\le\left\Vert g_{c'}\right\Vert _{\infty}\le p_{\min}$, we have $\mathrm{Cov}(g_{c'}^{*(l)},f)^{2}\le C\delta'\mu(f)\mathrm{Ent}(f)$ for some constant $C=C'/p_{\min}^{2}$. Plug it into (\ref{eq:cov}), we get $\mathrm{Ent}[\mu_{\tilde{T}}(f)]\le C\delta'\mathrm{Ent}(f).$
\end{proof}
\begin{lem}
	[Lemma 5.4 of \cite{martinelli2004glauber}] Let $\{\Omega,\mathcal{F},\nu\}$ be a probability space and let $f_{1}$ be a mean-zero random variable such that $\Vert f\Vert_{\infty}\le1$ and $\nu[|f_{1}|>\delta]\le e^{-2/\delta}$ for some $\delta\in(0,1)$. Let $f_{2}$ be a probability density w.r.t $\nu$, i.e. $f_{2}\ge0$ and $\nu(f_{2})=1$. Then there exists a numerical constant $C'>0$ independent of $\nu,f_{1},f_{2}$ and $\delta$, such that $\nu(f_{1}f_{2})\le C'\delta\mathrm{Ent}_{\nu}(f_{2}).$
\end{lem}
Now, as the last step, letting $\alpha=C/p_{\min}^{2}$ for the constant $C$ in Theorem \ref{thm:Dror5.3true} finishes the proof of Theorem \ref{thm:Dror5.3}.

\section{Results for $k$-coloring \label{sec:coloring}}

In this section we prove Theorem \ref{thm:thm2}, for which it is enough to verify the connectivity condition $\mathcal{C}$, in particular to show $p_{l}^\mathrm{free}\to 1$, as $l\to\infty$. In fact for the coloring model, as we will show in a moment, a vertex can change to all $k$ states in one step if all its children can change to 2 or 3 states in one step. We will first formalize this idea by defining the ``type'' of a vertex and then analyze the recursion with this weaker notion.

Recall the definition that for given configuration $\sigma\in\Omega_{T}$ with $\sigma_{\rho}=c$, we say the root can change to color $c'$ in one step iff there exist a path $\sigma=\sigma^{0},\sigma^{1},\dots,\sigma^{n}\in\Omega_{T}^{\sigma}$ such that for each $i,\sigma^i,\sigma^{i+1}$ differs by only one vertex and $\sigma_{\rho}^i=c,i=0,1,\dots.n-1,\sigma_{\rho}^{n}=c'.$ Let $C(\rho)$ denote the set of colors the root can change to in one step (including the original color).
We define the type of root to be rigid/type 2/type 3 if $|C(\rho)|=1/=2/\ge3$ respectively. For general vertex $x\in T$, not necessarily the root, we can similarly define $C(x)$ and rigid/type 2/type 3 by taking $x$ as the root of subtree $T_{x}$ and considering $\sigma|_{T_{x}}$. The set $C(x)$ is a function of $ \sigma_{T_x}$ and is independent of the rest of the tree.

Let $p_{l}^{r}=\mu_{l}(\mbox{the root is rigid})$, where $\mu_{l}$ is the Gibbs measure on $l$-level tree with free boundary condition.
Define $p_{l}^{(2)},p_{l}^{(3)}$ similarly, we have $p_{l}^{r}+p_{l}^{(2)}+p_{l}^{(3)}=1$. For tree $T$ with $l'>l$ levels and $x\in T$ that is $l$ levels above the bottom boundary, noting $\mu_{l'}|_{T_{x}}=\mu_{l}$, we also have
\[ \mu_{l'}(x\mbox{ is rigid/type 2/type 3})=\mu_{l'}|_{T_{x}}(x\mbox{ is rigid/type 2/type 3})=p_{l}^{r}/p_{l}^{(2)}/p_{l}^{(3)}. \]
The definition above is independent of the parent of $x$. In order to recursively analysing these probabilities, we need one further definition describing how the type of a vertex affects the type of its parent. For a given configuration $\sigma\in\Omega_{T}$ and $x\in\tilde{T}=T\backslash\{\rho\}$, recall $p(x)$ is the parent of $x$, we say $x$ is bad iff $C(x)\backslash\{\sigma_{p(x)}\}=\{\sigma_{x}\}$. Otherwise we say $x$ is good which implies $|C(x)\backslash\{\sigma_{p(x)}\}|\ge2$, i.e. $x$ has at least one more choice other than $\sigma_{p(x)}$. Note that the event that $x$ is bad depends only on $\sigma|_{T_{p(x)}}$ and given $\sigma_{x}$, for $x_i\in L_{x,1}$, events $\{x_i \mbox{ is bad}\}$ are conditionally i.i.d. and independent of the configurations outside $T_{x}$. Hence, by similar argument, we can define $p_{l}^{b}=1-p_{l}^{g}=\mu_{l'}(x\mbox{ is bad})$. The relation of rigid/type 2/type 3 and good/bad is shown in the following lemma.
\begin{lem}
	For $l'>l>0$ and $x\in T$ $l$ levels above the bottom boundary,
	\begin{equation}
		\mu_{l'}(x\mbox{ is bad}\mid x\mbox{ is rigid)=1},\mu_{l'}(x\mbox{ is bad}\mid x\mbox{ is type 2)=}\frac{1}{k-1},\mu_{l'}(x\mbox{ is bad}\mid x\mbox{ is type 3)=0},\label{eq:f23}
	\end{equation}
	Hence $p_{l}^{b}=p_{l}^{r}+\frac{1}{k-1}p_{l}^{(2)},p_{l}^{g}=p_{l}^{(3)}+\frac{k-2}{k-1}p_{l}^{(2)}.$
\end{lem}
\begin{proof}
	The first and third equation of (\ref{eq:f23}) is obvious as $|C(x)|$ and $|C(x)\backslash\{\sigma_{p(x)}\}|$ differs at most by one, and the equality about $p_{l}^{b}$ and $p_{l}^{g}$ follows immediately from the (\ref{eq:f23}). Hence we only need to show the second equation. Given $|C(x)|=2$, $x$ is bad iff $\sigma_{p(x)}\in C(x)$. So the conditional probability can be written as $P(C(x)=\{\sigma_{p(x)},\sigma_{x}\}\mid|C(x)|=2)$.
	
	Note $C(x)$ only depends on $T_{x}$, in particular it is conditionally independent of $\sigma_{p(x)}$ given $\sigma_{x}$. By symmetry, the distribution of $C(x)\backslash\{\sigma_{x}\}$ given $|C(x)|$ and $\sigma_{x}$ is uniformly distribution with probability $1/\binom{k-1}{|C(x)|-1}$. Hence
	\[ \Pr(C(x)=\{\sigma_{p(x)},\sigma_{x}\}\mid|C(x)|=2)=\frac{1}{\binom{k-1}{1}}=\frac{1}{k-1}. \]
\end{proof}
The next lemma follows a similar argument to Claim \ref{lem:freetoconnect} and Lemma \ref{lem:cnctcond}, and shows that in order to bound the probability of a vertex being free, it is enough to bound the probability of being bad.
\begin{lem}
	Suppose $k\ge4$. For any $\sigma\in\Omega_{T}$ and $x\in T$, if every child of $x$ is good, then $x$ is free.\label{lem:gdtofree}
\end{lem}
\begin{proof}
	Fix $c\in[k]$. Then since each child $y_i$ of $x$ is good there exist $c_i\in C(y_i)\backslash\{c,\sigma_{x}\}$. So to change $x$ from $\sigma_{x}$ to $c$ in one step, we can first change the color of every $y_i$ to $c_i$ in one step and then in the final step change $x$ from $\sigma_{x}$ to $c$. Hence $x$ is free.
\end{proof}
\medskip

Now we will show that for large enough $k$, in the region of non-reconstruction, the probability of seeing a bad vertex $l$ levels above bottom decays double exponentially fast in $l$. In fact we will prove the result for a region slightly larger than the known non-reconstruction region , which is $d\le k[\log k+\log\log k+\beta],$ for any $\beta<1-\ln 2$ (see \cite{sly2009reconstruction}).
\begin{thm}
	\label{thm: pbl} Suppose $\beta<1$, For sufficiently large $k$ and $d\le k[\log k+\log\log k+\beta]$, there exists a constant $l_{0}$ depending only on $k$ and $d$, such that for $l\ge l_{0},$
	\begin{equation}
		p_{l}^{b}\le\exp(-(k/2)^{l-l_{0}}).\label{eq:pbl_dbexp}
	\end{equation}
\end{thm}
Combining Lemma \ref{lem:gdtofree} and Theorem \ref{thm: pbl}, we can now finish the proof of Theorem \ref{thm:thm2}.
\begin{proof}
	[Proof of Theorem \ref{thm:thm2}] By Theorem~\ref{thm:thm1} and the known bounds on the reconstruction threshold for colorings~\cite{sly2009reconstruction} it is enough to show that the connectivity condition holds. The first part of the condition is obviously true for $k\ge 4$. For the second condition,
	\[ 1-p_{l}^\mathrm{free}=\Pr\nolimits_{\sigma\sim\mu_{l}}(\mbox{root is not free})\le\Pr\nolimits(\exists x\in L_{1},x\mbox{ is bad})\le dp_{l-1}^{b}\le d\exp(-(\frac{1}{2}k)^{l-l_{0}}), \]
	which tends to 0 as $l$ tends to infinity completing the proof.
\end{proof}
The proof of Theorem \ref{thm: pbl} is split into two phases, when $p_{l}^{b}$ is close to 1 and when $p_{l}^{b}$ is smaller than $\frac{1}{ed}$.
\begin{lem}
	Under the assumption of Theorem \ref{thm: pbl}, there exist a constant $l_{0}$ depending only on $k$ and $d$ such that $p_{l_{0}}^{b}<\frac{1}{ed}$.\label{lem:existl0}
\end{lem}
\begin{proof}
	This proof is similar to Lemma 2 and Lemma 4 of \cite{sly2009reconstruction}. We recursively analyze the probabilities as a function of the depth of the tree $l$. For $l=0$, $T$ consist only the bottom boundary and hence $p_{0}^{r}=1,p_{0}^{(2)}=p_{0}^{(3)}=0,p_{0}^{b}=p_{0}^{r}+\frac{1}{k-1}p_{0}^{(2)}=1$.
	
	For $l\ge 1,$ suppose without loss of generality that the color of the root is 1 and its children are $x_{1},\dots,x_d\in L_{1}$. Let $\mathcal{F}$ denote the sigma-field generated by $\{\sigma_{x_i},i=1,\dots,d\}$ and let $d_{c}=\#\{i,\sigma_{x_i}=c\}$ be the number of children with color $c$ for $2\le c\le k$. By definition, the sizes $|C(x_i)|$ and hence the type of $x_i$ are independent of $\mathcal{F}$ and are i.i.d.~distributed. Conditioning on $\mathcal{F}$ and $(|C(x_i)|)_{i=1}^{k},$ the set $C(x_i)\backslash\{\sigma_{x_i}\}$ is uniformly randomly chosen from all subsets of $[k]\backslash\{\sigma_{x_i}\}$ with $(|C(x_i)|-1)$ elements. Therefore the number of bad vertices of color $c$ given $\mathcal{F}$ is distributed as $\mathrm{Bin}(d_{c},p_{l-1}^{b})$.
	
	Following similar argument of Lemma \ref{lem:gdtofree}, the root can change to color $c$ in one step if and only if none of the $x_i$'s with color $c$ is bad, which happens with probability $(1-p_{l-1}^{b})^{d_{c}}$. Therefore we have
	\[ p_{l}^{b}=p_{l}^{r}+\frac{1}{k-1}p_{l}^{(2)}=\prod_{c=2}^{k}\mathbb{E}\left[1-(1-p_{l-1}^{b})^{d_{c}}\right]+\frac{1}{k-1}\sum_{c'=2}^{k}\mathbb{E}\bigg[ (1-p_{l-1}^{b})^{d_{c'}}\prod_{c\neq c'}\left(1-(1-p_{l-1}^{b})^{d_{c}}\right)\bigg]. \]
	Viewing the right hand side as a function of $(d_{2},\dots,d_{k})$, increasing $d_{c}$ means adding more vertices of color $c$, which increases the probability of blocking the move of the root. Therefore $p_{l}^{b}$ is an increasing function w.r.t every $d_{c}$. By symmetry, $(d_{2},\dots,d_{k})$ follows a multi-nominal distribution. Fix $\beta<\beta^{*}<1$, let $\tilde{d}_{c}$ be i.i.d. Poisson($D$) random variables where $D=\log k+\log\log k+\beta^{*}$. We can couple $(d_{2},\dots,d_{k})$ and $(\tilde{d}_{2},\dots,\tilde{d}_{k})$ such that $(d_{2},\dots,d_{k})\le(\tilde{d}_{2},\dots,\tilde{d}_{k})$ whenever $\sum_{c=2}^{k}\tilde{d}_{c}\ge d$. Letting $p=\Pr(\mbox{Poisson}((k-1)D)<d)$, the recursion relationship satisfies
	\begin{align*}
		p_{l}^{b} & =p_{l}^{r}+\frac{1}{k-1}p_{l}^{(2)}=\prod_{c=2}^{k}\mathbb{E}\left[1-(1-p_{l-1}^{b})^{d_{c}}\right]+\frac{1}{k-1}\sum_{c'=2}^{k}\mathbb{E}\bigg[(1-p_{l-1}^{b})^{d_{c'}}\prod_{c\neq c'}\left(1-(1-p_{l-1}^{b})^{d_{c}}\right)\bigg]\\
		& \le\prod_{c=2}^{k}\mathbb{E}\left[1-(1-p_{l-1}^{b})^{\tilde{d}_{c}}\right]+\frac{1}{k-1}\sum_{c'=2}^{k}\mathbb{E}(1-p_{l-1}^{b})^{\tilde{d}_{c'}}\prod_{c\neq c'}\mathbb{E}\left[1-(1-p_{l-1}^{b})^{\tilde{d}_{c}}\right]+p\\
		& =\left(1-\exp(-p_{l-1}^{b}D)\right)^{k-1}+\frac{k-1}{k-1}\exp(-p_{l-1}^{b}D)\left(1-\exp(-p_{l-1}^{b}D)\right)^{k-2}+p\\
		& =\left(1-\exp(-p_{l-1}^{b}D)\right)^{k-2}+p\le\exp\left(-(k-2)\exp(-p_{l-1}^{b}D)\right)+p
	\end{align*}
	where the last step follows from the fact $(1-r)^{k}\le e^{-kr}$ for $0<r<1$.
	
	The rest of the proof resembles the argument of Lemma 3 of \cite{sly2009reconstruction}. Let $f(x)=\exp\left(-(k-2)\exp(-xD)\right)+p$, $f(x)$ is an increasing function in $x$. So let $y_{0}=p_{0}^{b}=1$, $y_{l}=f(y_{l-1})$, we have $p_{l}^{b}\le y_{l}$, for all $l\ge0$. It is enough to show the existence of an $l_{0}$ such that $y_{l_{0}}\le\frac{1}{ed}$. Since $\frac{d}{dx}\exp(-x)\mid_{x=0}=-1$, there exist $\epsilon,\delta>0$ such that for $0<x<\delta$, $e^{-x}\le1-(1-\epsilon)x$. For large enough $k$ such that $(k-2)\exp(-D)=\frac{k-2}{k\log k}e^{-\beta^{*}}<\delta$, we have
	\[ y_{1}=f(1)\le 1-(1-\epsilon)\frac{k-2}{k\log k}e^{-\beta^{*}}+p. \]
	Since $\beta<\beta^{*}<1$, $(k-1)D-d\ge(\beta^{*}-\beta)k+o(k)$, by Hoeffding's inequality, $p=\exp(-\Omega(\frac{k}{\sqrt{d}}))=o(k^{-2})=o(d^{-1}).$ Therefore, for large enough $k$,
	\[ y_{1}\le 1-\frac{1-\epsilon}{2e\log k}+o(k^{-1})\le 1=y_{0}. \]
	Hence $y_{l}$ is a decreasing sequence as long as $(k-2)\exp(-y_{l}D)<\delta$. Choosing $\epsilon$ small enough, there exists $r'>r>0$ such that $(1-\epsilon)e^{-\beta^{*}}>e^{-1}(1+r')$, and we have
	\begin{align*}
		1-y_{l+1} & \ge 1-\left(p+1-(1-\epsilon)(k-2)\exp(-y_{l}D)\right)\\
		& \ge(1-\epsilon)\frac{(k-2)e^{-\beta^{*}}}{k\log k}\exp((1-y_{l})\log k)-p\\
		& \ge\frac{k-2}{k}(1-\epsilon)e^{1-\beta^{*}}(1-y_{l})-p\\
		& \ge(1+r')(1-y_{l})-p\ge(1+r)(1-y_{l})
	\end{align*}
	where the second last inequality follows from inequality $e^{x}>ex$, and the last inequality follows from that $1-y_{l}\ge 1-y_{1}=O(\frac{1}{\log k})$ while $p=o(k^{-2})$. Therefore after a constant number of steps $(k-2)\exp(-y_{l}D)\ge\delta$. Now let $e^{-\delta}<\alpha'<\alpha<1$, for some constants $\alpha,\alpha'$. When $k$ is large enough, $y_{l+1}\le p+e^{-\delta}<\alpha'<1$. Then again for $k$ large enough, $\exp(-y_{l+1}D)\ge\exp(-\alpha'D)\ge\exp(-\alpha\log k)=k^{-\alpha}.$ Therefore for $k$ large enough
	\[ y_{l+2}\le p+\exp(-(k-2)\exp(-y_{l+1}D))\le p+\exp(-\frac{1}{2}k^{1-\alpha})\le\frac{1}{ed}. \]
\end{proof}
After first $l_{0}$ levels, we cannot use the same method because the error of Poisson coupling becomes non-negligible; but meanwhile, $p_{l}^{b}$ is small enough such that bounding the total number of bad children is enough to finish the proof.
\begin{proof}
	[Proof of Theorem \ref{thm: pbl}] In order for a vertex to be bad, there must be at least $k-2$ of its children which are bad. Therefore,
	\[ p_{l}^{b}\le\binom{d}{k-2}(p_{l-1}^{b})^{k-2}\le(dp_{l-1}^{b})^{k-2}. \]
	Letting $l_{0}$ be the constant in Lemma \ref{lem:existl0} we complete the proof by induction on $l$ for $l\ge l_{0}$. If $l=l_{0}$, $p_{l_{0}}^{b}\le\frac{1}{ed}\le\frac{1}{e}$. If $p_{l}^{b}$ satisfies (\ref{eq:pbl_dbexp}) for $l>l_{0}$, then for $k$ large enough such that $\log(2k\log k)\le\frac{1}{6}k$ and $k-2\ge\frac{3}{4}k$,
	\begin{align*}
		p_{l+1}^{b} & \le(dp_{l}^{b})^{k-2}\le\left(2k\log k\exp\left(-(k/2)^{l-l_{0}}\right)\right)^{k-2}=\exp(k-2)\left(-(k/2)^{l-l_{0}}+\log(2k\log k)\right)\\
		& \le\exp\left(-\frac{3}{4}k\cdot\frac{2}{3}(\frac{1}{2}k)^{l-l_{0}}\right)=\exp\left(-(k/2)^{l+1-l_{0}}\right).
	\end{align*}
	Therefore (\ref{eq:pbl_dbexp}) holds for all $l\ge l_{0}$.
\end{proof}

\bibliography{ref} \end{document}